\documentclass[12pt]{amsart}
\usepackage{amssymb,latexsym}
\usepackage{enumerate}

\makeatletter
\@namedef{subjclassname@2010}{%
  \textup{2010} Mathematics Subject Classification}
\makeatother

\usepackage{mathrsfs}

\newtheorem{thm}{Theorem}[section]
\newtheorem{cor}[thm]{Corollary}
\newtheorem{lem}[thm]{Lemma}

\theoremstyle{definition}
\newtheorem{defin}[thm]{Definition}
\newtheorem{rem}[thm]{Remark}

\newtheorem{cl}{Claim}
\newcommand{\ndc}{nowhere dense closed }

\begin{document}
\baselineskip=17pt

\title[$F_{\sigma \delta}$- and
$G_{\delta \sigma}$-spaces]{Non-separable $h$-homogeneous absolute
$F_{\sigma \delta}$-spaces and $G_{\delta \sigma}$-spaces}

\author[S. V. Medvedev]{Sergey Medvedev}
\address{Faculty of Mechanics and Mathematics, South Ural State
University, pr. Lenina, 76,  454080 Chelyabinsk, Russia}
\email{medv@math.susu.ac.ru}

\date{}

\begin{abstract}
Denote by $Q(k)$ a $\sigma$-discrete metric weight-homogeneous
space of weight $k$. We give an internal description of the space
$Q(k)^\omega$. We prove that the Baire space $B(k)$ is densely
homogeneous with respect to $Q(k)^\omega$ if $k > \omega$.
Properties of some non-separable $h$-homogeneous absolute
$F_{\sigma \delta}$-sets and $G_{\delta \sigma}$-sets are
investigated.
\end{abstract}



\maketitle

\section*{Introduction}

All topological spaces under discussion are metrizable and
strongly zero-dimensional (i.e., Ind$X = 0$).

The aim of this paper is to characterize the space $Q(k)^\omega$
for $k > \omega$ (see Theorem~\ref{t-qR}). Combining this result
with general theorems about $h$-homogeneous spaces, we describe
some non-separable $h$-homogeneous absolute $F_{\sigma
\delta}$-sets and $G_{\delta \sigma}$-sets. We deal mainly with
the non-separable case, because the similar theorems for separable
spaces are already obtained.

About one hundred years ago, an arithmetical example of a strictly
$F_{\sigma \delta}$-subset of the space $\mathscr{N}$ of
irrationals was constructed by Baire~\cite{Bai}. Let us recall his
method. For every $n$-tuple $t \in \omega^n$, where $n \in
\omega$, take a perfect subset $X_t$ of $\mathscr{N}$. Suppose,
for all $t \in \omega^n$ and $i \in \omega$, $X_{t\hat{\,}i}$ is a
\ndc subset of $X_t$ and the union $ \cup \{X_{t\hat{\,}i}: i \in
\omega \}$ is dense in $X_t $. Put $E_n = \cup \{X_{t}: t \in
\omega^n \}$. Then the intersection $E = \cap \{E_n: n \in \omega
\}$ is a strictly $F_{\sigma \delta}$-subset of $\mathscr{N}$
(see~\cite{Bai}). The set $E$, which is obtained in such a way, is
called  \textit{a canonical element of class} 3. This notion was
introduced by Lusin. Keldysh (see~\cite{Ke34}, \cite{Ke44}) showed
that any two canonical elements of class 3 are homeomorphic.
Clearly, every canonical element of class 3 is a space of first
category.

Sierpinski observed that the property of being an absolute
$F_{\sigma \delta}$-set is internal. Consider a separable space
$X$. Let $\{X_t: t \in \omega^n, n \in \omega \}$ be a family
consisting of closed subsets of $X$ such that $X = \cup \{X_i: i
\in \omega \}$ and every $X_t = \cup \{X_{t\hat{\,}i}: i \in
\omega \}$. Suppose, for every $\chi \in \omega^\omega$ and an
arbitrary point $x_n \in X_{\chi {\upharpoonright}n}$, the
sequence $\{x_n: n \in \omega \}$ converges to a point $x \in X$.
Sierpinski~\cite{Sier} proved that under these conditions the
space $X$ is an absolute $F_{\sigma \delta}$-set. In a sense, the
Sierpinski result is a converse of the Baire theorem.

Using the Sierpinski method, van Engelen~\cite{EngQ} obtained a
topological characterization of the space $Q^\omega$, where $Q$ is
the space of rationals. Namely, he showed that every separable
absolute $F_{\sigma \delta}$-set of first category, which is
nowhere absolute $G_{\delta \sigma}$, is homeomorphic to
$Q^\omega$. This implies that $Q^\omega$ is a canonical element of
class 3. Moreover, van Engelen~\cite{EngQ} proved that the Cantor
set $\mathscr{C}$ is densely homogeneous with respect to
$Q^\omega$. In~\cite{En-tQ} van Engelen described all homogeneous
Borel sets which are either an $F_{\sigma \delta}$-subsets or a
$G_{\delta \sigma}$-subsets of the Cantor set $\mathscr{C}$, but
not both (see also~\cite{En-coun} and~\cite{FE}).

In the paper we use (and modify to the non-separable case) the
above theorems and a technique from Saint Raymond~\cite{Sai}, van
Mill~\cite{vm}, van Engelen~\cite{EngQ}, \cite{En-tQ}, and
Ostrovsky~\cite{Ostr}.

The paper is organized as follows: in Section 1 we prove some
lemmas to obtain a \textit{good} representation of an absolute
$F_{\sigma \delta}$-set (see Lemma~\ref{ad}). In Section 2 we give
a description of the space $Q(k)^\omega$ for $k > \omega$ (see
Theorem~\ref{t-qR}). In Section 3 we investigate the products
$Q(\tau) \times Q(k)^\omega$ and their dense complements in the
Baire space $B(\tau)$, where $ \omega \leq k \leq \tau$.

\section{Notation and some lemmas}

For all undefined terms and notation see~\cite{Eng}.

$X \approx Y$ means that $X$ and $Y$ are homeomorphic spaces. A
\emph{clopen} set is a set which is both closed and open. A
strongly zero-dimensional space $X$ is called
\textit{$h$-homogeneous} (or \emph{strongly homogeneous},
see~\cite{vm}, \cite{En-B2}) if every nonempty clopen subset of
$X$ is homeomorphic to $X$. Every $h$-homogeneous space is
homogeneous, but the converse does not hold. A space $X$ is called
\textit{weight-homogeneous} if for every nonempty open subset $U
\subseteq X$ we have $w(U) =w(X)$. For a cardinal $k$ put
$\mathscr{E}_k = \{X: \, X$ is a weight-homogeneous space of
weight $k$ and  $\mathrm{Ind}X= 0 \}$. Clearly, every
$h$-homogeneous space $X \in \mathscr{E}_k$ for $k = w(X)$.

Let $\mathscr{P}$ be a topological property.  A space $X$ is
\textit{nowhere} $\mathscr{P}$ if no nonempty open subset of $X$
has property $\mathscr{P}$. Let $\mathscr{P}_1$ and
$\mathscr{P}_2$ be topological properties; we write $X \in
\mathscr{P}_1 + \mathscr{P}_2$ if $X = A \cup B$, where $A$ has
property $\mathscr{P}_1$ and $B$ has  property $\mathscr{P}_2$.

Let $\mathscr{U}$ be a family of subsets of a metric space $(X,
\varrho)$. Put $\cup \mathscr{U} = \cup \{U: U \in \mathscr{U} \}$
and mesh$\,\mathscr{U}$ = sup\{diam($U$): $U \in \mathscr{U} \}$.
Denote by $[\mathscr{U}]$ the family $\{\mbox{cl}_X U: U \in
\mathscr{U}\}$. If $f:X\rightarrow Y$ is a mapping, then we write
$f(\mathscr{U}) = \{f(U): U \in \mathscr{U}\}$.

Let $k \geq \omega$. In~\cite{mq} we proposed to consider the
space $Q(k)$ as a non-separable analogue of weight $k$ for the
space $Q$ of rational numbers. By definition, $Q(k)$ is a
$\sigma$-product of $\omega$ copies of the discrete space $D$ of
cardinality $k$ with a basic point $(0,0,\ldots) \in k^\omega$,
i.e.,
$$Q(k) = \{(x_0, x_1, \ldots) \in k^\omega: \exists n \forall m
(m\geq n \Rightarrow x_m = 0 )\}.$$ Clearly, $Q(\omega) \approx
Q$. A topological description of the space $Q(k)$ is given by the
following theorem (see~\cite{mq}).

\begin{thm}\label{tq}
Let $X$ be a $\sigma$-discrete metric space of weight $k$ that is
homogeneous with respect to weight. Then $X$ is homeomorphic to
$Q(k)$.
\end{thm}

\begin{cor}\label{c-hQk}
The space $Q(k)$ is $h$-homogeneous for every cardinal $k$.
\end{cor}

\begin{lem}~\cite{os1}\label{eh}
Let $A_i$ be a \ndc subset of a metric strongly zero-dimensional
space $(X_i, \varrho_i)$, $i \in \{1, 2\}$. Let $X_1 \setminus
A_1$ and $X_2 \setminus A_2$ be homeomorphic $h$-homogeneous
spaces. If $f_0 : A_1 \rightarrow A_2$ is a homeomorphism, then
there exists a homeomorphism $f:X_1 \rightarrow X_2$ such that the
restriction $f{|}A_1 = f_0$.
\end{lem}

\begin{lem}\label{q2}
Let $A_i$ be a \ndc subset of an $h$-homogeneous space $(X_i,
\varrho_i)$, $i \in \{1, 2\}$. Let $g: X_1 \rightarrow X_2 $ be a
homeomorphism such that $g(A_1) = A_2$. Then for a homeomorphism
$f_0: A_1 \rightarrow A_2 $  and $ \varepsilon > 0$ satisfying
$\varrho_2(g{|}A_1, f_0) < \varepsilon$  there exists a
homeomorphism $f:X \rightarrow X$ such that the restriction
$f{|}A_1 = f_0$ and $\varrho_2(g, f) < \varepsilon$.
\end{lem}

\begin{proof} For $X_1 = X_2 = \mathscr{C}$, where
$\mathscr{C}$ is the Cantor set, the lemma was proved by van
Engelen~\cite[Theorem 2.2]{EngQ} (see also~\cite{vm}).

Suppose $X_1$ and $X_2$ are not compact. Put $\delta = \varepsilon
- \varrho_2(g{|}A_1, f_0)> 0.$ Take a discrete clopen cover
$\mathscr{U}$ of $X_1$ by sets of diameter less than $\delta/4$.
Let $\mathscr{U}_{\,2} = \{U^2_t: t \in T\}$ be a discrete clopen
cover of $X_2$ such that $\mathscr{U}_{\,2}$ refines
$g(\mathscr{U})$ and mesh$\,\mathscr{U}_{\,2} < \delta/4$. Put
$U^1_t = g^{-1}(U^2_t)$. Then $\mathscr{U}_{\,1} = \{U^1_t: t \in
T\}$ is a discrete clopen cover of $X_1$ and
mesh$\,\mathscr{U}_{\,1} < \delta/4$.

Let $T_* = \{t \in T: U^1_t \cap A_1 \neq \emptyset\}$. Put $V^i_t
= U^i_t \cap A_i$, where $i \in \{ 1, 2\}$ and $t \in T_*$. The
family $\mathscr{V}_2$ of nonempty intersections $ \{V^2_s \cap
f_0 (V^1_t): s \in T_*, t \in T_* \}$ forms a discrete clopen
(relative to $A_2$) cover of $A_2$. Hence, the family
$\mathscr{V}_1 = \{f^{-1}_0 (V_2): V_2 \in \mathscr{V}_2 \}$ is a
discrete clopen (relative to $A_1$) cover of $A_1$.

For every $V_1 \in \mathscr{V}_1$ there exists exactly one $t \in
T_*$ such that $V_1 \subseteq V^1_t \subset U^1_t$. Then $V_2 =
f_0(V_1) \subseteq V^2_s \subset U^2_s$ for some (unique) $s \in
T_*$. Since Ind$X_1 = 0$ and $V^1_t$ is closed in $U^1_t$, there
exists a retraction $r^1_t: U^1_t \rightarrow V^1_t$. The set
$V_1$ is clopen in $V^1_t$; this implies that $W_1 =
(r^1_t)^{-1}(V_1)$ is a clopen subset of $U^1_t$ and $W_1 \approx
X_1$. Similarly, there exists a retraction $r^2_s: U^2_s
\rightarrow V^2_s$, and $W_2 = (r^2_s)^{-1}(V_2) \approx X_2$. The
set $V_i$ is nowhere dense and closed in $W_i$ for $i \in \{ 1,
2\}$. Therefore, $W_1 \setminus V_1$ and $W_2 \setminus V_2$ are
homeomorphic $h$-homogeneous spaces. By virtue of Lemma~\ref{eh}
there exists a homeomorphism $f_{W_1}: W_1 \rightarrow W_2$ such
that the restriction $f_{W_1}{|}V_1 = f_0$ and $f_{W_1}(V_1) =
V_2$. By construction, diam($W_2$)$ < \delta/4$.

Take two points $a \in V_1$ and $x \in W_1$. Then
$$\begin{array}{l}
\varrho_2(g(x), f_{W_1}(x)) \leq \varrho_2(g(x), g(a)) +
\varrho_2(g(a), f_{W_1}(a)) + \varrho_2(f_{W_1}(a), f_{W_1}(x))
\leq \\ \qquad{} \leq \delta/4 + \varrho_2(g(a), f_0(a)) +
\delta/4 \leq \varrho_2(g{|}A_1, f_0) + \delta/2 < \varepsilon.
\end{array}$$

Let us construct the mapping $f: X_1 \rightarrow X_2$. Take a
point $x \in X_1$. If $x \in U^1_t$ for $t \in T \setminus T_*$,
then we define $f(x) = g(x)$. If $x \in U^1_t$, where $t \in T_*$,
then $r^1_t(x) \in V_1 \subseteq V^1_t$ for some $V_1 \in
\mathscr{V}_1$. Put $f(x) = f_{W_1}(x)$ provided $W_1 =
(r^1_t)^{-1}(V_1)$. In both cases $\varrho_2(g(x), f(x)) <
\varepsilon$. Clearly, $f$ is a homeomorphism and $f$ extends
$f_0$.
\end{proof}

\begin{lem}\label{q3}
Let $A_i$ be a \ndc subset of an $h$-homogeneous space $(X_i,
\varrho_i)$, $i \in \{1, 2\}$. Suppose we have a homeomorphism
$g:X_1 \rightarrow X_2 $ such that $g(A_1) = A_2$. Then for a
homeomorphism $f_0: A_1 \rightarrow A_2 $  and $ \varepsilon > 0$
satisfying $\varrho_1(g^{-1}{|}A_2, f^{-1}_0) + \varrho_2(g{|}A_1,
f_0) < \varepsilon$ there exists a homeomorphism $f:X_1
\rightarrow X_2$ such that the restriction $f{|}A_1 = f_0$ and
$\varrho_1(g^{-1}, f^{-1}) + \varrho_2(g, f) < \varepsilon$.
\end{lem}

\begin{proof}
For $d_1 = \varrho_1(g^{-1}{|}A_2, f^{-1}_0)$ and $d_2 =
\varrho_2(g{|}A_1, f_0)$ there exist $\varepsilon_1$ and
$\varepsilon_2$ such that $\varepsilon_1 + \varepsilon_2 =
\varepsilon$, $d_1 < \varepsilon_1$, and $d_2 < \varepsilon_2$.
Let $\delta = \mbox{min} \{\varepsilon_1 - d_1, \varepsilon_2 -
d_2\}$. Take the mapping $f: X_1 \rightarrow X_2$ that was
obtained in the proof of Lemma~\ref{q2}. It can easily be checked
that $f$ is the required homeomorphism.
\end{proof}

We identify cardinals with initial ordinals; in particular,
$\omega  = \{0, 1, 2, \ldots \}$. Denote by $\Lambda$ the unique
0-tuple, i.e., $k^0 = \{\Lambda \}$. Let $k^1 = k$. If $t \in
k^n$, then lh$(t)= n$ is a length of $t$ and $t\hat{\,} t_n =
(t_0, \ldots, t_{n-1}, t_n) \in k^{n+1}$ for $t_n \in k$ . Let
$k^{<\omega} = \cup \{k^i: i \in \omega\}$. For every infinite
cardinal $k$ let $B(k) = k^\omega$ be the Baire space of weight
$k$ (see~\cite{Eng}). For a point $x = (x_0, x_1, \ldots) \in
k^{\omega}$ we denote by $x {\upharpoonright} n$ the $n$-tuple
$(x_0, x_1, \ldots, x_{n-1})$. For an $n$-tuple $t \in k^n$ we
denote by $B_t (k)$ the Baire interval $\{x \in k^\omega: x
{\upharpoonright} n = t \}$.

A set $Y \subset X$ is \textit{a Souslin set} (or \textit{an
analytic set}~\cite{Han}, or an $A$-\textit{set}) in a space $X$
if, for some collection  $\{ F(t{\upharpoonright} n): n \in
\omega, t \in \omega^\omega \}$ of closed subsets of $X$, we have
$Y = \cup \{\bigcap \{F(t{\upharpoonright} n): n \in \omega \} : t
\in \omega^\omega \}$. $Y$ is called a \textit{co-Souslin} set in
a space $X$ if $X \setminus Y$ is a Souslin set in $X$. If
$\mathscr{A}$ and $\mathscr{B}$ are two families of subsets of
$X$, we say (see~\cite{Han}) that $\mathscr{B}$ is a \textit{base}
for $\mathscr{A}$ if each member of $\mathscr{A}$ is a union of
members from  $\mathscr{B}$. A \textit{$\sigma$-discrete base} is
a base which is also a $\sigma$-discrete family. A mapping $f: X
\rightarrow Y$ is called \textit{co-$\sigma$-discrete} if the
image of each discrete family in $X$ has a $\sigma$-discrete base
in $Y$.

If $A$ is an $F_\sigma$-subset of $X$, then we write $A \in
\mathscr{F}_\sigma(X)$. Similarly, we introduce families
$\mathscr{G}_\sigma(X)$, $\mathscr{F}_{\sigma \delta}(X)$, and so
forth. A space $Y$ is \textit{an absolute $G_\delta$-set} if for
every homeomorphic embedding $f:Y \rightarrow X$ of $Y$ in a space
$X$ the image $f(Y) \in \mathscr{G}_\sigma(X)$. Denote by
$\mathscr{G}_\delta$, $\mathscr{F}_{\sigma \delta}$, and
$\mathscr{G}_{\delta \sigma}$ the families of  absolute
$G_\delta$-sets, absolute $F_{\sigma \delta}$-sets, and absolute
$G_{\delta \sigma}$-sets, respectively.

For a space $X$ define the family $LF(X) = \{Y:$ each point $y \in
Y$ lies in some clopen neighborhood that is homeomorphic to a
closed subset of $X\}$. A space $Y \in \sigma LF(X)$ if $Y = \cup
\{Y_i: i \in \omega \}$, where each $Y_i \in LF(X)$ and $Y_i$ is a
closed subset of $Y$. We say that $X$ is \textit{locally of
weight} ${<}k$ if each point from $X$ has a neighborhood of weight
${<}k$. $X$ is said to be $\sigma$-locally of weight ${<}k$ (in
symbols, $X \in \sigma LW({<}k)$) provided $X = \cup \{X_i: i \in
\omega \}$, where each $X_i$ is locally of weight ${<}k$
(see~\cite{St}).

\begin{lem}\label{sg}
\rm{1)} Let a space $X$ be locally $\sigma LW({<}k) +
\mathscr{G}_{\delta \sigma}$. Then $X \in \sigma LW({<}k) +
\mathscr{G}_{\delta \sigma}$.

\rm{2)} Let $ X_i \in \sigma LW({<}k) + \mathscr{G}_{\delta
\sigma}$ for each $i \in \omega$. Then $\cup \{X_i: i \in \omega
\} \in \sigma LW({<}k) + \mathscr{G}_{\delta \sigma}.$
\end{lem}

\begin{proof}
For each point $x \in X$ take a neighborhood $U(x)$ such that
$U(x) = L(x) \cup G(x)$, where $L(x)$ is $\sigma LW({<}k)$ and
$G(x) \in \mathscr{G}_{\delta \sigma}$. Then $L = \cup\{L(x): x
\in X \}$ is $\sigma LW({<}k)$ (see~\cite{St}) and $G =
\cup\{G(x): x \in X \}$ belongs to $\mathscr{G}_{\delta \sigma}$
(see~\cite{Eng}, Pr.~4.5.8). So, $X = L \cup G \in \sigma LW({<}k)
+ \mathscr{G}_{\delta \sigma}$.

The second part of the lemma is obvious.
\end{proof}

\begin{lem}\label{hg}
Suppose an absolute Souslin set $Y$ is not $\sigma LW({<}k) +
\mathscr{G}_{\delta \sigma}$ for a cardinal $k$. There exists a
\ndc subset $Z \subset Y$ such that $Z$ is nowhere $\sigma
LW({<}k) + \mathscr{G}_{\delta \sigma}$.
\end{lem}

\begin{proof}
As an absolute Souslin set, $Y$ has the Baire property. Therefore,
$Y = F \cup G$, where $F$ is of first category in $Y$, $G \in
\mathscr{G}_\delta$, and $F \cap G = \emptyset$. Clearly, $F
\notin \sigma LW({<}k) + \mathscr{G}_{\delta \sigma}$ and $F \in
\mathscr{F}_\sigma(Y)$. Then $F = \cup \{F_i: i \in \omega \}$,
where each $F_i$ is a \ndc subset of $Y$. This implies that there
exists a $j \in \omega$ such that $F_j$ is not $\sigma LW({<}k) +
\mathscr{G}_{\delta \sigma}$. From Lemma~\ref{sg} it follows that
the nonempty set
$$Z = F_j \setminus \cup\{U \; \mathrm{is \; open \; in} \;
F_j : U \in \sigma LW({<}k) + \mathscr{G}_{\delta \sigma} \}$$ is
closed in $F_j$ and nowhere $\sigma LW({<}k) + \mathscr{G}_{\delta
\sigma}$. Thus, $Z$ is the required set.
\end{proof}

The following Theorem~\ref{hf} is important. Using it, for a given
non-$\sigma LW({<}k) + \mathscr{G}_{\delta \sigma}$-space $A$, we
can choose a \ndc subset of $A$ which is nowhere $\sigma LW({<}k)
+ \mathscr{G}_{\delta \sigma}$.

\begin{thm}\label{hf}
Let $F$ be an $F_\sigma$-subset of the Baire space $B^*(k)$ and $A
\subset F$. Suppose  $A$ is a co-Souslin set in $B^*(k)$ and $A$
is not $\sigma LW({<}k) + \mathscr{G}_{\delta \sigma}$. Then there
exists a closed subset $M \subseteq B^*(k)$ satisfying the
following conditions

\begin{enumerate}[\upshape 1)]
    \item $A \cap M$ is a \ndc subset of $A$;
    \item $A \cap M$ is nowhere $\sigma LW({<}k) +
     \mathscr{G}_{\delta \sigma}$ and of first category;
    \item $M \subseteq F$ and $M$ is homeomorphic to $B^*(k)$.
\end{enumerate}
\end{thm}

\begin{proof} For $k = \omega$ the theorem was proved by van
Engelen (see \cite[Lemma 3.2]{EngQ}), because $B^*(\omega)$
denotes the Cantor set $\mathscr{C}$.

Let the cardinal $k > \omega$. Denote by $X$ the Baire space
$B(k)$. Replacing $F$ by $F \cap \overline{A}$, we may assume that
$F \subseteq \overline{A}$, where the bar denotes the closure in
$X$.

Clearly, $F {\setminus} A$ is a Souslin set in $X$. Let us show
that $w(F {\setminus} A)= k$. Indeed, if $w(F {\setminus} A) < k$,
then the boundary of $A$ in $F$ has weight ${<} k$ and the
interior of $A$ in $F$ is an $\mathscr{F}_\sigma$-subset of $X$.
This contradicts the condition $A \notin \sigma LW({<}k) +
\mathscr{G}_{\delta \sigma}$.

By virtue of the Hansell theorem (see~\cite[Theorem 4.1]{Han})
there exists a continuous co-$\sigma$-discrete mapping $\varphi:
B(k) \rightarrow F {\setminus} A$ of the Baire space $B(k)$ onto
$F {\setminus} A$. Denote this copy of $B(k)$ by $H$, i.e., we
have $\varphi: H \rightarrow F {\setminus} A$. An open subset $U
\subset H$ is called \textit{usual} if there exists a set $E_U \in
\mathscr{F}_\sigma(X)$ such that $\varphi(U) \subseteq E_U
\subseteq F$ and $E_U \cap A$ is $\sigma LW({<}k) +
\mathscr{G}_{\delta \sigma}$ or $E_U \cap A = \emptyset$. The set
$\widehat{T} = \cup \{U: U$ is usual in $H \}$ is open in $H$. Let
us show that $\widehat{T}$ is usual in $H$. Take a
$\sigma$-discrete cover $\tau$ of $\widehat{T}$ by usual sets.
Since $\varphi$ is co-$\sigma$-discrete, there exists a
$\sigma$-discrete base $\mathscr{B}$ for $\varphi(\tau)$ in $X$.
To each $B \in \mathscr{B}$ assign $U_B \in \tau$ such that $B
\subseteq \varphi(U_B)$. Put $B^* = E_{U_B} \cap \mbox{cl}_X B$.
Then $\{B^*: B \in \mathscr{B} \}$ is  a $\sigma$-discrete family
in $X$. Clearly, $E = \cup \{B^* : B \in \mathscr{B} \} \subseteq
F$ and $E \in \mathscr{F}_\sigma(X)$. According to Lemma~\ref{sg},
$A \cap E \in \sigma LW({<}k) + \mathscr{G}_{\delta \sigma}$. By
construction, $\varphi(\widehat{T}) \subseteq \cup \mathscr{B}
\subseteq E$. Hence, the set $\widehat{T}$ is usual in $H$.

We claim that $F \setminus (A \cup E) \neq \emptyset$. Assume the
converse. Then $F {\setminus} E \subseteq A$. Since $A \cap E \in
\sigma LW({<}k) + \mathscr{G}_{\delta \sigma}$ and $F \setminus E
\in \mathscr{G}_{\delta \sigma}$, we see that $$A = (A \cap E)
\cup (F \setminus E) \in \sigma LW({<}k) + \mathscr{G}_{\delta
\sigma},$$ a contradiction. Thus, $F \setminus (A \cup E) \neq
\emptyset$. Hence, the set $T = \varphi^{-1}(F \setminus E)$ is
nonempty. Clearly, $T \cap \widehat{T} = \emptyset$.

\begin{cl}\label{c1}
If $\emptyset \neq U$ is open in $T$, then $A \cap
\overline{\varphi(U)}  \notin \sigma LW({<}k) +
\mathscr{G}_{\delta \sigma}$.
\end{cl}

Take an open subset $U$ of $T$. Let $U = U_1 \cap T$, where $U_1$
is open in $H$. Then $$\varphi(U_1)= \varphi(U)\cup \varphi(U_1
\setminus U) \subset (\overline{\varphi(U)} \cap F) \cup E =
E_1.$$ Clearly, $E_1 \subseteq F$. Since $\overline{\varphi(U)}
\cap F \in \mathscr{F}_\sigma(X)$, we have $E_1 \in
\mathscr{F}_\sigma(X)$. By construction, $U_1$ is not a usual set.
Therefore, $A \cap E_1 \notin \sigma LW({<}k) +
\mathscr{G}_{\delta \sigma}$. But $A \cap E  \in \sigma LW({<}k) +
\mathscr{G}_{\delta \sigma}$ and $A \subset F$; hence, $A \cap
\overline{\varphi(U)} $ is not $\sigma LW({<}k) +
\mathscr{G}_{\delta \sigma}$. Claim~\ref{c1} is verified. \hfill
$\lozenge$

Since $E \in \mathscr{F}_\sigma(F)$, we see that $T =
\varphi^{-1}(F \setminus E) \in \mathscr{G}_\sigma(H)$. Hence, $T$
is homeomorphic to a closed subset of the Baire space $H= B(k)$.
We may now assume that $T$ is a closed subset of the Baire space
$B(k)$ with the standard metric $d$, and the restriction of $d$ to
$T$ induces the original topology on $T$. Let $\varrho$ be the
standard metric on the Baire space $X = B(k)$.

Let $F = \cup \{F_i: i \in \omega \}$, where each $F_i$ is closed
in $X$. By induction, for every $n \in \omega$, $t \in \omega^n$,
and $\alpha \in k^n$ we shall construct  sets $Z(t, \alpha)
\subset X$, clopen sets $V(t, \alpha) \subset X$ and $W(t, \alpha)
\subset X$, disjoint families $\mathscr{V}(t, \alpha) = \{
V(t\hat{\,}i, \alpha\hat{\,}\alpha_i): i \in \omega, \alpha_i \in
k \}$ in $X$, clopen sets $T(t, \alpha) \subset T$, disjoint
families $\mathscr{T}(t, \alpha) = \{ T(t\hat{\,}i,
\alpha\hat{\,}\alpha_i): i \in \omega, \alpha_i \in k \}$ in $T$,
families $\mathscr{R}(t, \alpha) = \varphi(\mathscr{T}(t,
\alpha))$, and families $\mathscr{U}(t\hat{\,}i, \alpha)$, where $
i \in \omega$, of clopen subsets of $X$ such that

\begin{enumerate}[\upshape ({a}1)]
    \item $Z(t, \alpha)$ is nowhere $\sigma LW({<}k) +
      \mathscr{G}_{\delta \sigma}$ and nowhere dense in $X$;
    \item $Z(t, \alpha)$ is a \ndc subset of
        $A \cap \overline{\varphi(T(t, \alpha))}$;
    \item $\varphi(T(t, \alpha)) \subset V(t, \alpha)$ and
      $\overline{Z(t, \alpha)} \subset F_j \bigcap \overline{\varphi
      (T(t, \alpha))} \subset V(t, \alpha)$   for some $j \in \omega$;
    \item $\mathrm{cl}_X(\cup \mathscr{V}(t, \alpha))=
      \overline{Z(t, \alpha)}
      \cup \bigcup \mathscr{V}(t, \alpha)$ and $\overline{Z(t, \alpha)}
      \cap \bigcup \mathscr{V}(t,\alpha)= \emptyset$;
    \item $W(t, \alpha) \subset V(t, \alpha)$, $W(t, \alpha)
       \cap A \neq \emptyset$, and $W(t, \alpha)
       \cap \bigcup \mathscr{U}(t\hat{\,}0, \alpha) = \emptyset$;
    \item for every $i \in \omega$ the family $\{U \cap
       \overline{Z(t, \alpha)}:
       U \in \mathscr{U}(t\hat{\,}i, \alpha)\}$ forms a clopen cover
        of $\overline{Z(t, \alpha)}$ such that
       $\cap \{\cup \mathscr{U}(t\hat{\,}i, \alpha): i \in \omega \} =
       \overline{Z(t, \alpha)}$ and mesh$\mathscr{U}(t\hat{\,}i,
       \alpha) \leq (n +i+1)^{-1}$;
    \item for every $U \in \mathscr{U}(t\hat{\,}(i+1), \alpha))$
       there exists $U_1 \in \mathscr{U} (t\hat{\,} i, \alpha)$
       such that $U \subset U_1$;
    \item for every $U \in \mathscr{U}(t\hat{\,} i, \alpha)$,
       where $i \in \omega$, the set $U \setminus \bigcup \mathscr{U}
       (t\hat{\,}(i+1), \alpha)$ contains a nonempty
       clopen subset  ${\vartriangle }U$ such that
       ${\vartriangle} U \cap \varphi(T(t, \alpha)) \neq \emptyset$
       and $V(t \hat{\,}i, \alpha \hat{\,}\alpha_{n+1}) \subset
       {\vartriangle }U$ for every $\alpha_{n+1} \in k$;
    \item $V(t, \alpha) \cap V(s, \beta) = \emptyset$ whenever
       $(t, \alpha)\neq (s, \beta)$ for $s \in
       \omega^n$ and $ \beta \in k^n$;
    \item diam$(V(t, \alpha)) \leq (n+1)^{-1}$ (with respect to the
          metric $\varrho$ on $X$);
    \item diam$(T(t, \alpha)) \leq (n+1)^{-1}$ (with respect to the
          metric $d$ on $T$);
    \item $\cup \mathscr{T}(t, \alpha) \subset T(t, \alpha)$;
    \item the family $\{V(t\hat{\,}i, \alpha\hat{\,}\alpha_{n+1}):
      \alpha_{n+1} \in k \}$ is discrete in $X$ for every
       $i \in \omega$;
    \item the family $\{T(t\hat{\,}i, \alpha\hat{\,}\alpha_{n+1}):
      \alpha_{n+1} \in k \}$ is discrete in $T$ for
      every $i \in \omega$;
   \item cl$_X(\cup \mathscr{R}(t, \alpha))= \overline{Z(t, \alpha)}
     \cup \bigcup \{\overline{\varphi(T(t \hat{\,}i, \alpha\hat{\,}
     \alpha_{n+1})}:  i \in \omega,  \alpha_{n+1} \in k \}$.
\end{enumerate}

Let $V(\Lambda, \Lambda) = X$ and $T(\Lambda, \Lambda) = T$. By
Claim~\ref{c1}, the set $A \cap \overline{\varphi(T(\Lambda,
\Lambda))} $ is not $\sigma LW({<}k) + \mathscr{G}_{\delta
\sigma}$. Since
$$A \cap \overline{\varphi(T)}  = \cup \{A \cap \overline{\varphi(T)}
\cap F_i : i \in \omega \},$$ there exists a $j \in \omega$ such
that the set $A \cap \overline{\varphi(T)} \cap  F_j $ is not
$\sigma LW({<}k) + \mathscr{G}_{\delta \sigma}$. By
Lemma~\ref{hg}, the set $A \cap \overline{\varphi(T)} \cap  F_j $
contains a nowhere dense closed subset $Z(\Lambda, \Lambda)$ such
that $Z(\Lambda, \Lambda)$ is nowhere $\sigma LW({<}k) +
\mathscr{G}_{\delta \sigma}$. Clearly, $Z(\Lambda, \Lambda)$ is a
weight-homogeneous space of weight $k$. Choose a nonempty clopen
set $W(\Lambda, \Lambda) \subset X \setminus \overline{Z(\Lambda,
\Lambda)}$ with $W(\Lambda, \Lambda) \cap A \neq \emptyset$.

Since Ind$X = 0$, there exists a retraction $r: (V(\Lambda,
\Lambda)\setminus W(\Lambda, \Lambda)) \rightarrow
\overline{Z(\Lambda, \Lambda)}$. Let $\{O_i: i \in \omega \}$ be a
decreasing sequence of clopen sets such that $\overline{Z(\Lambda,
\Lambda)} \subset O_i$ and $\varrho(x, \overline{Z(\Lambda,
\Lambda)}) < (i+1)^{-1}$ for each point $x \in O_i$. Consider  a
sequence $\{\tau_i: i \in \omega \}$  of clopen (in
$\overline{Z(\Lambda, \Lambda)}$) covers of $\overline{Z(\Lambda,
\Lambda)}$ such that mesh$(\tau_i) < (i+1)^{-1}$, $\tau_{i+1}$
refines $\tau_i$, and $|\tau_i| = k$. For every $i \in \omega$ the
clopen family $\mathscr{U}(i, \Lambda) = \{O_i \cap r^{-1}(B): B
\in \tau_i \}$ is discrete in $X$ and mesh$\mathscr{U}(i, \Lambda)
< (i+1)^{-1}$. Since $Z(\Lambda, \Lambda)$ is a nowhere dense
subset of $\overline{\varphi(T(\Lambda, \Lambda))}$, the
intersection $U \cap \left( \varphi(T(\Lambda, \Lambda))\setminus
\overline{Z(\Lambda, \Lambda)} \right)$ is nonempty for each $U
\in \mathscr{U}(i, \Lambda)$. Without loss of generality we may
assume that, for all $i \in \omega$ and $U \in \mathscr{U}(i,
\Lambda)$, the set $U \setminus \bigcup \mathscr{U}(i+1, \Lambda)
= U {\setminus} O_{i+1}$ contains a nonempty clopen subset
${\vartriangle }U$ such that ${\vartriangle }U \cap
\varphi(T(\Lambda, \Lambda)) \neq \emptyset$. The set
$\varphi(T(\Lambda, \Lambda))$ is not locally of weight ${<}k$.
Then there exists a discrete family $\mathscr{V}_{{\vartriangle
}U}$ such that $\cup \mathscr{V}_{{\vartriangle }U} \subset
{\vartriangle }U$,
 $\mathscr{V}_{{\vartriangle }U}$ consists of $k$ clopen sets, and
$V \cap \varphi(T(\Lambda, \Lambda)) \neq \emptyset$ for each $V
\in \mathscr{V}_{{\vartriangle }U}$. Since the family
$\mathscr{U}(i, \Lambda)$ is discrete in $X$ for fixed $i$, the
family $\mathscr{V}_i = \{V: V \in \mathscr{V}_{{\vartriangle }U}$
for some $U \in \mathscr{U}(i, \Lambda)\}$ is discrete too. Let
$\mathscr{V}_i = \{V(i, \alpha) : \alpha \in k \}$. Clearly, $V(i,
\alpha) \cap V(j, \beta) = \emptyset$ whenever $(i, \alpha) \neq
(j, \beta)$.  Put $\mathscr{V}(\Lambda, \Lambda) = \{V(i, \alpha)
:  i \in \omega, \alpha \in k\}$.

Let us show that (a4) is satisfied. By construction, for each $i
\in \omega$ and every $V \in \mathscr{V}_i$ we have $V \cap
\bigcup \mathscr{U}(i+1, \Lambda) = \emptyset$. This implies that
$V \cap \overline{Z(\Lambda, \Lambda)} = \emptyset$. Then $
\overline{Z(\Lambda, \Lambda)}  \cap \bigcup\mathscr{V}(\Lambda,
\Lambda)= \emptyset$. From (a6) and (a8) it follows that
$\overline{Z(\Lambda, \Lambda)} \subset
\overline{\cup\mathscr{V}(\Lambda, \Lambda)}$. Now, consider a
point $x \in \mbox{cl}_X( \cup \mathscr{V}(\Lambda, \Lambda))
\setminus \overline{Z(\Lambda, \Lambda)}$. Choose the smallest
number $j \in \omega$ such that $x \notin \cup \mathscr{U}(j+1,
\Lambda)$. Then $x \in {\vartriangle}U$ for some $U \in
\mathscr{U}(j, \Lambda)$. The family $\mathscr{V}_j$ is discrete
and locally closed; therefore, $x \in {\vartriangle}U \cap \bigcup
\mathscr{V}_j $. So, $x \in \cup \mathscr{V}(\Lambda, \Lambda)$.

We can find a nonempty clopen set $T(i, \alpha) \subset
\varphi^{-1}(V(i, \alpha)) \subset T$ of diameter less than 1. Put
$\mathscr{T}(\Lambda, \Lambda) = \{T(i, \alpha): i \in \omega,
\alpha \in k\}$ and $\mathscr{R}(\Lambda, \Lambda) =
\varphi(\mathscr{T}(\Lambda, \Lambda))$.

Clearly, (a1)--(a15) are satisfied for $n = 0$.

Now suppose that $Z(t, \alpha)$, $V(t, \alpha)$, $\mathscr{V}(t,
\alpha)$,  $T(t, \alpha)$, $\mathscr{T}(t, \alpha)$,
$\mathscr{R}(t, \alpha)$, and $\mathscr{U}(t\hat{\,}i, \alpha)$,
where $i \in\omega $,  have been defined for any $t \in \omega^m$
and $\alpha \in k^m$ with $m \leq n$. Fix $t \in \omega^n$ and
$\alpha \in k^n$. By Claim~\ref{c1} the set $A \cap
\overline{\varphi(T(t, \alpha))}$ is not $\sigma LW({<}k) +
\mathscr{G}_{\delta \sigma}$. Since
$$A \cap \overline{\varphi(T(t, \alpha))} = \cup
\{A \cap \overline{\varphi(T(t, \alpha))}  \cap F_i : i \in \omega
\},$$ there exists a $j \in \omega$ such that the set $A \cap
\overline{\varphi(T(t, \alpha))} \cap F_j $ is not $\sigma
LW({<}k) + \mathscr{G}_{\delta \sigma}$. By Lemma~\ref{hg}, it
contains  a \ndc subset $Z(t, \alpha)$ such that $Z(t, \alpha)$ is
nowhere $\sigma LW({<}k) + \mathscr{G}_{\delta \sigma}$ and $Z(t,
\alpha) \in \mathscr{E}_k$. Choose a nonempty clopen set $W(t,
\alpha) \subset V(t, \alpha) \setminus \overline{Z(t, \alpha)}$
with $W(t, \alpha) \cap A \neq \emptyset$.

Take a retraction $r: (V(t, \alpha)\setminus W(t, \alpha))
\rightarrow \overline{Z(t, \alpha)}$. Let $\{O_i: i \in \omega \}$
be a decreasing sequence of clopen sets such that $\overline{Z(t,
\alpha)} \subset O_i$ and $\varrho(x, \overline{Z(t, \alpha)}) <
(n+i+1)^{-1}$ for each point $x \in O_i$. Consider a sequence
$\{\tau_i: i \in \omega \}$  of clopen covers of $\overline{Z(t,
\alpha)}$ such that mesh$(\tau_i) < (n+i+1)^{-1}$, $\tau_{i+1}$
refines $\tau_i$, and $|\tau_i| = k$. For every $i \in \omega$ the
clopen family $\mathscr{U}(t\hat{\,}i, \alpha) = \{O_i \cap
r^{-1}(B): B \in \tau_i \}$ is discrete in $X$. By construction,
 mesh$\mathscr{U}(t\hat{\,}i, \alpha) < (n+i+1)^{-1}$. As
above, we may assume that (a8) holds, i. e., for all $i \in
\omega$ and $U \in \mathscr{U}(t\hat{\,}i, \alpha)$, the clopen
set ${\vartriangle }U$ contains a discrete family
$\mathscr{V}_{{\vartriangle }U}$ such that $\cup
\mathscr{V}_{{\vartriangle }U} \subset {\vartriangle }U$,
$\mathscr{V}_{{\vartriangle }U}$ consists of $k$ clopen sets, and
$V \cap \varphi(T(t, \alpha)) \neq \emptyset$ for each $V \in
\mathscr{V}_{{\vartriangle }U}$. The family $\mathscr{V}_i = \{V:
V \in \mathscr{V}_{{\vartriangle }U}$ for some $U \in
\mathscr{U}(t\hat{\,}i, \alpha)\}$ is discrete; reindex this
family as $\mathscr{V}_i = \{V(t\hat{\,}i, \alpha \hat{\,}
\alpha_{n+1}) : \alpha_{n+1} \in k \}$. Put $\mathscr{V}(t,
\alpha) = \{V(t\hat{\,}i, \alpha \hat{\,}\alpha_{n+1}) : i \in
\omega, \alpha_{n+1} \in k\}$.

Take a clopen set $T(t\hat{\,}i, \alpha \hat{\,}\alpha_{n+1})
\subset \varphi^{-1}(V(t\hat{\,}i, \alpha \hat{\,}\alpha_{n+1}))$
of diameter less than $(n+1)^{-1}$. Put $\mathscr{T}(t, \alpha) =
\{T(t\hat{\,}i, \alpha\hat{\,}\alpha_{n+1}): i \in \omega,
\alpha_{n+1} \in k\}$ and $\mathscr{R}(t, \alpha) =
\varphi(\mathscr{T}(t, \alpha))$.

All conditions (a1)--(a15) are satisfied. This completes the
induction.

For each $n \in \omega$ we define sets
$$\begin{array}{l}
Y_n = \cup \{Z(t, \alpha): t \in \omega^i, \alpha \in k^i, i \leq
n \},   \\
Z_n = \cup \{\overline{Z(t, \alpha)}: t \in \omega^i, \alpha \in
k^i,  i \leq n\},  \\
P_n = \cup \{\overline{\varphi(T(t, \alpha))}: t \in \omega^n,
\alpha \in k^n \},
\end{array}$$
and $M_n = Z_n \cup P_{n+1} $. Let $Y = \cup\{Y_n : n \in \omega
\}$, $Z = \cup \{Z_n : n \in \omega \}$, and $M = \cap \{M_n : n
\in \omega \}$. Clearly, $Y \subset Z \subset M$. By (a2), $Y
\subset A$.

\begin{cl}\label{c2}
$Y$ is of first category and $Y$ is nowhere $\sigma LW({<}k) +
\mathscr{G}_{\delta \sigma}$.
\end{cl}

Indeed, from (a3) and (a4) it follows that $Z_n$ is a \ndc subset
of $Z_{n+1}$ for every $n \in \omega$. Then each $Z_n$ is a
nowhere dense subset of $Z$. Hence, $Z$ is of first category.
Therefore, $Y$ is of first category as a dense subset of the space
$Z$ of first category. Suppose that an open (in $Y$) set $U
\subset Y$ is $\sigma LW({<}k) + \mathscr{G}_{\delta \sigma}$. Let
us find the smallest number $n$ such that $U \cap (Y_n \setminus
Y_{n-1}) \neq \emptyset$. Then for some $t \in \omega^n$ and $
\alpha \in k^n$ the set $U \cap Z(t, \alpha)$ is $\sigma LW({<}k)
+ \mathscr{G}_{\delta \sigma}$, contradicting (a1). Hence, $Y$ is
nowhere $\sigma LW({<}k) + \mathscr{G}_{\delta \sigma}$.
Claim~\ref{c2} is verified. \hfill $\lozenge$

\begin{cl}\label{cl3}
$\overline{Y} = M \subset F$ and $Y = A \cap M$.
\end{cl}

First,  we show by induction that each $M_n$ is closed in $X$. The
set $M_0$ is closed by (a15) with $n = 0$. Suppose $M_{n-1}$ is
closed. Since $M_n \subset M_{n-1}$, we have  $\overline{M_n}
\subset M_{n-1}$. Take a point $x \in \overline{M_n}$. Clearly, if
$x \in Z_n$, then $x \in Z_n \subset M_n$. Let $x \in
\overline{M_n} \setminus Z_n$. Then $x \in \overline{P_{n+1}}$. On
the other hand, $$x \in \overline{\varphi(T(t, \alpha))} \subset
P_n \subset M_{n-1}$$ for some $t \in \omega^n$ and $ \alpha \in
k^n$. From (a12), (a3), (a9), and (a15) it follows that
$$\overline{\varphi(T(t, \alpha))} \cap \overline{P_{n+1}} =
\overline{Z(t, \alpha)} \cup \bigcup \{\overline{\varphi(T(t
\hat{\,}i, \alpha\hat{\,} \alpha_{n+1})}:  i \in \omega,
\alpha_{n+1} \in k \}.$$ The family $\{\overline{\varphi(T(t
\hat{\,}i, \alpha\hat{\,} \alpha_{n+1})}:  i \in \omega,
\alpha_{n+1} \in k \}$ is disjoint. Hence, for some $i \in \omega$
and $\alpha_{n+1} \in k$ we obtain $x \in \overline{\varphi(T((t
\hat{\,}i, \alpha\hat{\,} \alpha_{n+1})} \subset M_n$. So,
 $\overline{M_n} = M_n$.

Thus, $M = \cap \{M_n : n \in \omega \}$ is closed in $X$.

From (a3) and (a10) it follows that $Y$ is dense in $M$. Hence,
$\overline{Y} = M$.

To prove $Y = A \cap M$ it suffices to show that $ M \setminus Y
\subset F \setminus A$. Take a point $x \in M \setminus Y$. If $x
\in \overline{Z(t, \alpha)}$ for some $t \in \omega^n, \alpha \in
k^n$, from (a2) and (a3) it follows that $x \in F \setminus A$.
Consider the case $x \in M \setminus Z$. As above, for every $n
\in \omega$ there exist $t(n) \in \omega^n$ and $\alpha(n) \in
k^n$ such that $x \in \overline{\varphi(T(t(n), \alpha(n)))}$.
Using (a12) we can find $\chi \in \omega^\omega$ and $\xi \in
k^\omega$ such that $\chi{\upharpoonright}n = t(n)$ and
$\xi{\upharpoonright}n = \alpha(n)$ for every $n \in \omega$.
Then, by the Cantor theorem (see~\cite[Theorem 4.3.8]{Eng}), we
obtain that the intersection $\cap \{\overline{\varphi(T
(\chi{\upharpoonright}n, \xi{\upharpoonright}n))}: n \in \omega
\}$ consists of a one point $x$. Similarly, $ \cap \{T
(\chi{\upharpoonright}n, \xi{\upharpoonright}n) : n \in \omega\} =
\{c \}$. The mapping $\varphi$ is continuous. Therefore,
$\varphi(c) = x \in F \setminus A$. Thus, $ M \setminus Y \subset
F \setminus A$. Claim~\ref{cl3} is proved. \hfill $\lozenge$

Let us verify that $Y$ is nowhere dense in $A$. Take a point $y
\in Y$; then $y \in Z(t, \alpha)$ for some $t \in \omega^n$ and $
\alpha \in k^n$. For every $\varepsilon > 0$ there exists $i \in
\omega$ such that $(n+i+1)^{-1} < \varepsilon$. From (a5) it
follows that the $\varepsilon$-ball about $y$ contains the
nonempty open set $A \cap W(t \hat{\,}i, \alpha\hat{\,}
\alpha_{n+1})$ for some $\alpha_{n+1} \in k$. By construction,
$W(t \hat{\,}i, \alpha\hat{\,} \alpha_{n+1}) \cap Y = \emptyset$;
hence, $Y$ is nowhere dense in $A$.

Theorem~\ref{hf} is proved.
\end{proof}

\begin{lem}\label{as}
Let $k > \omega$. Let $F$ be an $F_\sigma$-subset of the Baire
space $B(k)$ and $A \subset F$. Suppose $B(k) \setminus A$ is a
Souslin set in $B(k)$, $A$ is of first category, and $A$ is
nowhere $\sigma LW({<}k) + \mathscr{G}_{\delta \sigma}$. Then for
every $\varepsilon > 0$ there exists a family $\mathscr{A} =
\{A(i, \alpha): i \in \omega, \alpha \in k\}$ of subsets of $A$
such that
\begin{enumerate}[\upshape (b1)]
    \item $A = \cup\mathscr{A}$ and $\mathrm{mesh}\mathscr{A}
    < \varepsilon$;
    \item each $A(i, \alpha)$ is of  first category and
      nowhere $\sigma LW({<}k) + \mathscr{G}_{\delta \sigma}$;
    \item each $A(i, \alpha)$ is a \ndc subset of $A$ and
       $\overline{A(i, \alpha)} \subset  F$;
    \item $\overline{A(i, \alpha)} \bigcap \overline{A(j, \beta)} =
       \emptyset$ whenever $(i, \alpha) \neq (j, \beta)$;
    \item for every $i \in \omega$ the family $\{A(i, \alpha):
       \alpha \in k\}$ is discrete in $B(k)$.
\end{enumerate}

The bar denotes the closure in $B(k)$.
\end{lem}

\begin{proof}
Let $F = \bigcup \{F_j: j \in \omega \}$, where each $F_j$ is
closed in $B(k)$. As a space of first category, $A = \bigcup
\{H_m: m \in \omega \}$, where each $H_m$ is a nowhere dense
closed subset of $A$. Take the nonempty intersections $F_j \cap
H_m$ and reindex these sets  as $\{C_i: i \in \omega \}$. Then $A
= \bigcup \{C_i: i \in \omega \}$, where each $C_i$ is a nowhere
dense closed subset of $A$, $C_i = \overline{C_i} \cap A$, and
$\overline{C_i} \subset F$.

Denote by $\varrho$ the standard complete metric on $B(k)$. Fix $i
\in \omega$. Take a retraction $r: \overline{A} \rightarrow
\overline{C_i}$. Let $\{O_j: j \in \omega \}$ be a decreasing
sequence of clopen sets from $B(k)$ such that  $\overline{C_i}
\subset O_j$ and $\varrho(x, \overline{C_i}) < (j+1)^{-1}$ for
each point $x \in O_j$. Consider a sequence $\{\tau_m: m \in
\omega \}$ of clopen (in $\overline{C_i}$) covers of
$\overline{C_i}$ such that mesh$(\tau_m) < (m+1)^{-1}$ and
$\tau_{m+1}$ refines $\tau_m$. For every $j \in \omega$ the family
$\mathscr{U}_j = \{O_j \cap r^{-1}(B): B \in \tau_j \}$ is
discrete in $B(k)$ and mesh$\mathscr{U}_j < (j+1)^{-1}$. Without
loss of generality we may assume that, for all $j \in \omega$ and
$U \in \mathscr{U}_j$, the set $U \setminus \bigcup
\mathscr{U}_{j+1}$ contains a clopen (in $\overline{A}$) subset
${\vartriangle }U$ with ${\vartriangle }U \cap A \neq\emptyset$.
Let $\mathscr{U} = \{U \in \mathscr{U}_j: j \in \omega \}$. By
Theorem~\ref{hf} for each $U \in \mathscr{U}$ there exists a set
$Y_U \subset {\vartriangle }U$ such that $Y_U$ is nowhere $\sigma
LW({<}k) + \mathscr{G}_{\delta \sigma}$, $Y_U$ is of first
category, $Y_U$ is a \ndc subset of $A$, and $\overline{Y_U}
\subset F$. Now, put $B_i = C_i \cup (\bigcup \{Y_U: U \in
\mathscr{U} \})$.

One easily verifies that $B_i$ is nowhere $\sigma LW({<}k) +
\mathscr{G}_{\delta \sigma}$, $B_i$ is of first category, $B_i$ is
 a \ndc subset of $A$, and $\overline{B_i} = \overline{C_i} \cup
\bigcup \{\overline{Y_U}: U \in \mathscr{U} \} \subset F$.

Clearly, $A = \cup \{B_i: i \in \omega \}$.

Let $\gamma_0$ be a discrete cover  of $\overline{B_0}$ by clopen
(in $\overline{B_0}$) sets such that mesh($\gamma_0) <
\varepsilon$. For each $i > 0$ the set $\overline{B_i} \setminus
\bigcup \{\overline{B_j}: j < i \}$ is open in $\overline{B_i}$;
it may be empty. Take a disjoint $\sigma$-discrete (in $B(k)$)
cover $\gamma_i$ of $\overline{B_i} \setminus \bigcup
\{\overline{B_j}: j < i \}$ by clopen (in $\overline{B_i}$) sets
such that mesh($\gamma_i) < \varepsilon$. Then the family $\gamma
= \{V \in \gamma_i: i \in \omega \}$ is $\sigma$-discrete in
$B(k)$. It consists of pairwise disjoint closed subsets of
diameter less than $\varepsilon$. Hence, $\gamma = \{ V(i,
\alpha): i \in \omega , \alpha \in k_i\}$, where each $k_i \leq k$
and each family $\{V(i, \alpha): \alpha \in k_i\}$ is discrete in
$B(k)$.

If $k_i < k$ for some $i$, we take a set $V(i, \alpha)$ and
separate $V(i, \alpha)$ into $k$ nonempty disjoint clopen (in
$V(i, \alpha)$) subsets. This can be done, because $k > \omega$
and $V(i, \alpha)$ is nowhere of weight ${<}k$. So, we may assume
that each $k_i = k$.

The sets $A(i, \alpha) = A \cap V(i, \alpha)$ form the required
family $\mathscr{A}$.
\end{proof}

\begin{lem}\label{ad}
 Let $k > \omega$. Let $A$ be an $F_{\sigma \delta}$-subset
  of the Baire space $B(k)$ such that $A$ is of first category
 and $A$ is nowhere $\sigma LW({<}k) + \mathscr{G}_{\delta \sigma}$.
  There exists a family
  $\{A(t, \alpha): t \in \omega^n$, $\alpha \in k^n, n \in \omega\}$
  of nonempty $F_{\sigma \delta}$-subsets of $B(k)$ such that
\begin{enumerate}[\upshape (c1)]
    \item $A = A(\Lambda, \Lambda)$, $ A(t, \alpha) =
       \bigcup \{A(t\hat{\,}i, \alpha\hat{\,} \alpha_{n+1}): i
       \in \omega, \alpha_{n+1} \in k \}$, and each $A(t\hat{\,}i,
       \alpha\hat{\,} \alpha_{n+1})$ is a \ndc  subset of
       $A(t, \alpha)$;
    \item each $A(t, \alpha)$ is of  first category and
      nowhere $\sigma LW({<}k) + \mathscr{G}_{\delta \sigma}$;
    \item the family $\{A(t\hat{\,}i, \alpha\hat{\,} \alpha_{n+1}):
       \alpha_{n+1} \in k \}$ is discrete in $B(k)$ for each $i
       \in \omega$;
    \item $\overline{A(t, \alpha)} \cap \overline{A(s, \beta)} =
       \emptyset$ whenever $(t, \alpha) \neq (s, \beta)$, if
       $s \in \omega^n$ and $\beta \in k^n$;
    \item $\overline{A(t, \alpha)}$ is  homeomorphic to $B(k)$;
    \item \textrm{\rm diam}$(A(t, \alpha)) < (n+1)^{-1}$;
    \item for all $\chi \in \omega^\omega$ and $\xi \in k^\omega$
      the intersection $\bigcap \{ A(\chi{\upharpoonright}i,
       \xi{\upharpoonright}i): i \in \omega \}$ is a one-point
       subset of $A$.
\end{enumerate}

By the bar we denote the closure in $B(k)$.
\end{lem}

\begin{proof}
Let $A = \cap \{F_i: i \in \omega \}$,  where each $F_i$ is an
$F_\sigma$-subset of $B(k)$.

We may assume that $F_0 = \overline{A}$. Put $A(\Lambda, \Lambda)
= A$.  By Lemma~\ref{as} for $\varepsilon =1/2$ and $F = F_1$ we
obtain $A(\Lambda, \Lambda) = \cup \{A(i, \alpha): i \in \omega,
\alpha \in k\}$, where sets $\{A(i, \alpha) \}$ satisfy (b1)--(b5)
and each $\overline{A(i, \alpha)} \subset F_1$. By virtue of the
Stone theorem, each $\overline{A(i, \alpha)}$ is homeomorphic to
$B(k)$. For $\varepsilon = 1/3$ and $F = F_2$ we again apply
Lemma~\ref{as} to each $A(i, \alpha)$. Then $A(i, \alpha) = \cup
\{A(i\hat{\,}j, \alpha \hat{\,} \beta): j \in \omega, \beta \in k
\}$, where each $\overline{A(i\hat{\,}j, \alpha \hat{\,} \beta)}
\subset F_2$, and so forth.

Let us verify (c7). Using (c1), (c6), and the Cantor theorem, we
have $\bigcap \{\overline{ A(\chi{ \upharpoonright}i,
\xi{\upharpoonright}i)}$: $i \in \omega \} = \{x\}$ for a point $x
\in B(k)$. Since $\overline{ A(\chi{ \upharpoonright}i,
\xi{\upharpoonright}i)} \subset F_i$, we obtain that $x \in \cap
\{F_i : i \in \omega \}$. Thus, $x \in A$.
\end{proof}

\section{A topological characterization of $Q(k)^\omega$}

In this Section we show (see Theorem~\ref{tqc}) that all spaces
from the family $\mathscr{X}_k$ are homeomorphic, where
$\mathscr{X}_k = \{X: X \in \mathscr{E}_k$, $X$ is of first
category, $X$ is an absolute $F_{\sigma \delta}$-set and nowhere
$\sigma LW({<}k) + \mathscr{G}_{\delta \sigma} \}$. By
Theorem~\ref{t-qR}, the space $Q(k)^\omega$ belongs to
$\mathscr{X}_k$. Corollary~\ref{c-hko} states that every absolute
$F_{\sigma \delta}$-space of weight ${\leq}k$ is homeomorphic to a
closed subset of $Q(k)^\omega$.

\begin{thm}\label{tqc}
Let $k > \omega$. Let $X_1$ and $X_2$ be dense $F_{\sigma
\delta}$-subsets of the Baire space $B(k)$ such that $X_1$ and
$X_2$ are sets of first category and nowhere $\sigma LW({<}k) +
\mathscr{G}_{\delta \sigma}$. There exists a homeomorphism $f:
B(k) \rightarrow B(k)$ such that $f(X_1) = X_2$.
\end{thm}

\begin{proof}
Consider two copies $B_1$ and $B_2$ of $B(k)$ with metrics
$\varrho_1$ and $\varrho_2$, respectively. Let $X_j \subset B_j$
and $\overline{X_j} = B_j$ for $j \in \{1, 2 \}$. Using
Lemma~\ref{ad}, we obtain two families $\{X_j(t, \alpha): t \in
\omega^i$, $\alpha \in k^i, i \in \omega\}$ of nonempty $F_{\sigma
\delta}$-subsets of $B_j$ satisfying (c1)--(c7).

To prove the theorem, for $j \in \{1, 2 \}$, we transform  the
family $\{X_j(t, \alpha) \}$ to a family $\{Z_j(s, \beta) \}$ such
that the family $\{Z_j(s, \beta) \}$ satisfies the conditions
(d1)--(d10) (see below) and $\{Z_j(s, \beta) \}$ refines $\{X_j(t,
\alpha) \}$ if $\mathrm{lh}(s) = \mathrm{lh}(t)$.

For $t \in \omega^i$, where $i \in \omega$, define $\nu(t) = t_0 +
t_1 + \ldots + t_{i-1} + i$. By definition, put $\nu(\Lambda) =
0$.

For $n \in \omega$ let $\omega^i_n = \{t \in \omega^i: \nu(t) = n
\}$. For example, $\omega^0_0 = \{\Lambda\}$, $\omega^1_1 = \{0
\}$,  and  $\omega^2_3 = \{(1,0), \, (0,1) \}$. Clearly, each set
$\omega^i_n$ is finite and $\omega^i_n = \emptyset$ whenever $n <
i$. It is easy to verify that $\omega^i_n \cap \omega^j_m =
\emptyset $ providing $(i, n) \neq (j, m)$.

For $j \in \{1, 2 \}$ and $t \in \omega^i$ put $X^*_j(t) = \bigcup
\{X_j(t, \alpha): \alpha \in k^i \}$.

For $j \in \{1, 2 \}$, $n \in \omega$, and $i \leq n$ we will
construct discrete (in $B_j$) families $\{Z_j(s, \beta)$: $s \in
\omega^i_n$, $\beta \in k^i\}$ of nonempty \ndc subsets of $X_j$,
sets $Z^*_j(s) = \bigcup \{Z_j(s, \beta): \beta \in k^i \}$,
functions $\varphi^i_{n, j}: \omega^i_n \times k^i \rightarrow
\omega^i$ and $\psi^i_{n, j}: \omega^i_n \times k^i \rightarrow
k^i$, and homeomorphisms $f_n: B_1 \rightarrow B_2$ such that

\begin{enumerate}[\upshape (d1)]
    \item $Z_j(s, \beta)$ is a clopen subset of
       $X_j(\varphi^i_{n, j}(s, \beta), \psi^i_{n, j}(s, \beta))$
       and $\overline{Z_j(s, \beta)}$ is a clopen subset of
       $\overline{X_j(\varphi^i_{n, j}(s, \beta), \psi^i_{n, j}
       (s, \beta))}$  for each $(s, \beta) \in \omega^i_n
       \times k^i$;
    \item $\overline{Z_j(s, \beta)} \bigcap \overline{Z_j(r,
    \alpha)} = \emptyset$ whenever $(s, \beta) \neq (r,
    \alpha)$ with $\mathrm{lh}(s) = \mathrm{lh}(r)$;
    \item $\overline{Z^*_j(s)} = \bigcup \{\overline{Z_j(s, \beta)}:
       \beta \in k^i \}$ and  $\overline{Z^*_j(s)} \bigcap
       \overline{Z^*_j(r)}  = \emptyset$ if
       $\mathrm{lh}(s) = \mathrm{lh}(r)$ and $s \neq r$;
    \item if $(s, \beta) = (r\hat{\,}l, \alpha\hat{\,} \gamma)  \in
      \omega^i_n \times k^i$, then $Z_j(s, \beta)$ is a \ndc  subset
      of $Z_j(r, \alpha)$ and $Z^*_j(s)$ is a  \ndc subset of $Z^*_j(r)$;
    \item $\nu(\varphi^i_{n, j}(s, \beta)) \geq n $ for
       each $s \in \omega^i_n$;
    \item $X^*_j(r) \subseteq \bigcup \{Z^*_j(s): s \in \omega^i_n,
       n \leq \nu(r)\}$ for each $r \in \omega^i$;
    \item $\overline{X^*_j(r)} \bigcap \overline{Z^*_j(s)}$
        is clopen in  $\overline{X^*_j(r)}$ if $\mathrm{lh}(r) =
        \mathrm{lh}(s)$;
    \item $\overline{Z_2(s, \beta)} = f_n(\overline{Z_1(s, \beta)})$
       and $\overline{Z^*_2(s)}= f_n(\overline{Z^*_1(s)})$  for
       each $(s, \beta) \in \omega^i_n \times k^i$;
    \item $ f_n(\overline{Z_1(s, \beta)}) = f_{n-1}(\overline{Z_1(s,
      \beta)})$ and $f_n(\overline{Z^*_1(s)})=f_{n-1}(\overline{Z^*_1(s)})$
       for each $(s, \beta) \in \omega^i_m \times k^i$ with $m <n$;
    \item $\varrho_2(f_n, f_{n-1})+\varrho_1(f^{-1}_n, f^{-1}_{n-1})
       < n^{-1}$ for each $n \geq 1$.
\end{enumerate}

Let $f_0: B_1 \rightarrow B_2$ be the identity mapping. For $j \in
\{1, 2 \}$ put $Z_j(\Lambda, \Lambda) = X_j(\Lambda, \Lambda) =
X_j$, $Z^*_j(\Lambda) = X^*_j(\Lambda) = X_j$, $\varphi^0_{0,
j}(\Lambda, \Lambda) = \Lambda$, and $\psi^0_{0, j}(\Lambda,
\Lambda) = \Lambda$.

Since the family $\{\overline{X_j(0, \alpha)}: \alpha \in k) \}$
is discrete in $B_j$, we can find a discrete clopen cover $\tau_j$
of $B_j$ such that  mesh$( \tau_j) < 1/2$ and every $U \in \tau_j$
intersects at most one set of the family
 $\{\overline{X_j(0, \alpha)}: \alpha \in k \}$, where $j \in
  \{1, 2 \}$. We may assume that $ \tau_2 = f_0( \tau_1)$.
Note that for any open subset $U$ of $B_j$ the conditions $U
\bigcap \overline{X_j(0, \alpha)} \neq \emptyset $ and $U \bigcap
X_j(0, \alpha) \neq \emptyset $ are equivalent. Let $$\tau^*_1 =
\{U \in \tau_1: U \cap X^*_1(0) \neq \emptyset \mathrm{\quad or
\quad} f_0(U) \cap X^*_2(0) \neq \emptyset\}.$$  The family
$\tau^*_1$ can be indexed as $\{U_1(\beta): \beta \in k \}$. Then
$$\tau^*_2 = f_0(\tau^*_1) = \{U_2( \beta) = f_0(U_1(\beta)): \beta
\in k \}.$$

For $j \in \{1, 2 \}$ and $\beta \in k$  choose the smallest
number $l_{\beta, j} \in \omega$ such that the intersection
$U_j(\beta) \bigcap X_j(l_{\beta, j}, \alpha_{\beta, j})$ is
nonempty for some $\alpha_{\beta, j} \in k$. Denote this
intersection by $Z_j(0, \beta)$. Then the number $l_{\beta, j} =
\varphi^1_{1, j}(0, \beta)$ and the ordinal $\alpha_{\beta, j} =
\psi^1_{1, j}(0, \beta)$ are assigned to $(0, \beta)$. Note that
in the case $U_j(\beta) \bigcap X^*_j(0) \neq \emptyset$ we have
$l_{\beta, j} = 0$, else $l_{\beta, j} \geq 1$. By construction,
each $Z_j(0, \beta)$ is a \ndc subset of $Z_j(\Lambda, \Lambda) =
X_j$, $Z_j(0, \beta)$ is clopen in $Z^*_j(0) = \bigcup \{Z_j(0,
\beta) $: $\beta \in k \}$, and $X^*_j(0) \subseteq Z^*_j(0)$.

The set $\cup \tau^*_j$ is clopen in $B_j$. Hence, for each $i \in
\omega$ the sets $\overline{X^*_j(i)} \setminus \bigcup \tau^*_j$
and $\overline{X^*_j(i)} \cap \bigcup \tau^*_j =
\overline{X^*_j(i)} \bigcap \overline{Z^*_j(0)}$  are clopen in
$\overline{X^*_j(i)}$.

By the Stone theorem, $\overline{Z_1(0, \beta)} \approx B(k)
\approx \overline{Z_2(0, \beta)}$ for every $\beta \in k$.
Lemma~\ref{eh} implies that there exists a homeomorphism
$g_{\beta}: U_1( \beta) \rightarrow U_2(\beta)$ such that
$g_{\beta} \left(\overline{Z_1(0, \beta)} \right) =
\overline{Z_2(0, \beta)}$. Since $\mathrm{diam}(U_1( \beta)) <
1/2$ and diam$(U_2( \beta)) < 1/2$, we have
$$\varrho_1(g_{\beta}^{-1}, f_0 ^{-1}|U_2( \beta)) +
\varrho_2(g_{\beta}, f_0|U_1( \beta)) < 1/2+1/2=1.$$

Consider a point $x \in B_1$. If $x \in U_1( \beta) \in \tau^*_1$,
then define $f_1(x) = g_{\beta}(x)$. For $x \in B_1 \setminus
\bigcup\tau^*_1$ put $f_1(x) = f_0(x)$. Clearly, $f_1: B_1
\rightarrow B_2$ is a homeomorphism, $\varrho_1(f_1^{-1},
f_0^{-1}) + \varrho_2(f_1, f_0) < 1$, and
$f_1(\overline{Z^*_1(0)}) = \overline{Z^*_2(0)}$.

So, the conditions (d1)--(d10) are satisfying for $n = 1 $.

Suppose that $Z_j(s, \beta)$, $f_n$, $\varphi^i_{n, j}$, and
$\psi^i_{n, j}$ have been obtained for $s$ with $\nu(s) \leq n$,
$\beta \in k^{\textrm{lh}(s)}$, $i \leq n$, and $j \in \{1, 2\}$.
Let us construct a homeomorphism $f_{n+1}: B_1 \rightarrow B_2$.
To do this we shall apply induction on $q$ for $s \in
\omega^q_{n+1}$, starting with $q = n+1$ and finishing with $q =
1$.

The first step $q = n+1$. The set $\omega^{n+1}_{n+1}$ contains a
unique tuple $s=(0, \ldots, 0)$. Let $\tilde{s} = s
{\upharpoonright}n$. Fix $\alpha \in k^n$. According to (d1),
$Z_j(\tilde{s}, \alpha)$ is a clopen subset of $X_j(\varphi^n_{n,
j}(\tilde{s}, \alpha), \psi^n_{n, j}(\tilde{s}, \alpha))$ for $j
\in \{1, 2\}$. Choose the smallest number $m_{\tilde{s}, \alpha
,j}$ such that the set  $Z_j(\tilde{s}, \alpha) \bigcap
X^*_j(\varphi^n_{n, j}(\tilde{s}, \alpha)\hat{\,}m_{\tilde{s},
\alpha, j})$ is nonempty. Take a discrete clopen (in
$\overline{Z_j(\tilde{s}, \alpha)}$) cover $\tau_j$ of
$\overline{Z_j(\tilde{s}, \alpha)}$ such that mesh$( \tau_j) <
(2n+2)^{-1}$ and every $U \in \tau_j$ intersects at most one set
of the family $$\{X_j(\varphi^n_{n, j}(\tilde{s},
\alpha)\hat{\,}m_{\tilde{s}, \alpha ,j}, \psi^n_{n, j}(\tilde{s},
\alpha)\hat{\,}\alpha_{n+1}): \alpha_{n+1} \in k \}.$$ By (d8), we
have $f_n(\overline{Z_1(\tilde{s}, \alpha)}) =
\overline{Z_2(\tilde{s}, \alpha)}.$ Then we may assume that $
\tau_2 = f_n( \tau_1)$. Denote by $\tau^*_1$ the family
$$\{U \in \tau_1: U \cap X^*_1(\varphi^n_{n, 1}(\tilde{s},
\alpha)\hat{\,}m_{\tilde{s}, \alpha ,1}) \neq \emptyset \mathrm{
\, or \,}  f_n(U) \cap X^*_2(\varphi^n_{n, 2}(\tilde{s},
\alpha)\hat{\,}m_{\tilde{s}, \alpha ,2}) \neq \emptyset\}.$$ Since
$X^*_1(\varphi^n_{n, 1}(\tilde{s}, \alpha)\hat{\,}m_{\tilde{s},
\alpha ,1})$ is a weight-homogeneous space of weight $k$, the
family $\tau^*_1$ can be indexed as $\{U_1(\gamma): \gamma \in k
\}$. Put $$\tau^*_2 = f_n(\tau^*_1) = \{U_2(\gamma) =
f_n(U_1(\gamma)): \gamma \in k \}.$$

For $j \in \{1, 2 \}$ and $\gamma \in k$  choose the smallest
number $l_{\gamma, j} \in \omega$ such that the set
$$Z_j(s, \alpha \hat{\,} \gamma) = U_j(\gamma) \bigcap
X_j(\varphi^n_{n, j}(\tilde{s}, \alpha )\hat{\,}l_{\gamma, j},
\psi^n_{n, j}(\tilde{s}, \alpha )\hat{\,} \alpha_{\gamma, j})$$ is
nonempty for some $\alpha_{\gamma, j} \in k$. Then the tuples
$$\varphi^n_{n, j}(\tilde{s}, \alpha )\hat{\,}l_{\gamma, j} =
\varphi^{n+1}_{n+1, j}(s, \alpha \hat{\,} \gamma) \in
\omega^{n+1}$$ and $$\psi^n_{n, j}(\tilde{s}, \alpha )\hat{\,}
\alpha_{\gamma, j} = \psi^{n+1}_{n+1, j}(s, \alpha \hat{\,}
\gamma) \in k^{n+1}$$ are assigned to $(s, \alpha \hat{\,}
\gamma)$.

Clearly, $l_{\gamma, j} \geq m_{\tilde{s}, \alpha ,j}$. Then
$$\nu(\varphi^{n+1}_{n+1, j}(s, \alpha \hat{\,} \gamma)) \geq
\nu(\varphi^n_{n, j}(\tilde{s}, \alpha)) + 1 + l_{\gamma, j} \geq
n+1.$$ By construction, each $Z_j(s, \alpha \hat{\,} \gamma)$ is a
\ndc subset of $Z_j(\tilde{s}, \alpha)$ and each $Z_j(s, \alpha
\hat{\,} \gamma)$ is clopen in $\cup \{Z_j(s, \alpha \hat{\,}
\gamma) : \gamma \in k \}$.

By the Stone theorem, $\overline{Z_1(s, \alpha \hat{\,} \gamma)}
\approx B(k) \approx \overline{Z_2(s, \alpha \hat{\,} \gamma)}$
for every $\gamma \in k$. Lemma~\ref{eh} implies that there exists
a homeomorphism $g_{\gamma}: U_1(\gamma) \rightarrow U_2(\gamma)$
such that $g_{\gamma} \left(\overline{Z_1(s, \alpha \hat{\,}
\gamma)} \right) = \overline{Z_2(s, \alpha \hat{\,} \gamma)}$.
Since diam$(U_1(\gamma)) < (2n+2)^{-1}$ and diam$(U_2(\gamma)) <
(2n+2)^{-1}$, we have
$$\varrho_1(g_{\gamma}^{-1}, f_n ^{-1}|U_2(\gamma)) +
\varrho_2(g_{\gamma}, f_n|U_1(\gamma)) < (2n+2)^{-1}+(2n+2)^{-1} =
(n+1)^{-1}.$$

Since $\tau_j = \{U_j(\gamma): \gamma \in k \}$ forms a discrete
clopen cover of $\overline{Z_j(\tilde{s}, \alpha)}$  for $j \in
\{1,2 \}$, we can obtain a homeomorphism $f^s_{n+1,\alpha}:
\overline{Z_1(\tilde{s}, \alpha)} \rightarrow
\overline{Z_2(\tilde{s}, \alpha)}$ by the rule
\[f_{n+1, \alpha}^s(x) = \left\{
    \begin{array}{rl}
    g_{\gamma}(x), & \mbox{if } x \in U_1(\gamma) \mbox{ for some }
     U_1(\gamma)  \in \tau^*_1 \\
    f_n(x), & \mbox{if } x \in \overline{Z_1(\tilde{s},
               \alpha)} \setminus \bigcup \tau^*_1
    \end{array} \right. \]

For $j \in \{1,2 \}$ put $Z^*_j(\tilde{s}) = \bigcup
\{Z_j(\tilde{s}, \alpha): \alpha \in k^n \}$ and $Z^*_j(s) =
\bigcup \{Z_j(s, \alpha \hat{\,} \gamma): \alpha \in k^{n+1} \}$.
Using (d2), we see that the combination mapping
$$f^s_{n+1} = \nabla \{f_{n+1, \alpha }^s: \alpha \in k^n
 \} : \overline{Z^*_1(\tilde{s})} \rightarrow \overline{Z^*_2(\tilde{s}
 )}$$ is a homeomorphism. By the inductive assumption, we have
$f_n(\overline{Z_1(\tilde{s}, \alpha)}) = \overline{Z_1(\tilde{s},
\alpha)}$ for each $(\tilde{s}, \alpha)$. Then for the tuple $s
=(0, \ldots, 0) \in \omega^{n+1}_{n+1}$ we obtain the following
inequality
$$\varrho_1 \left((f^s_{n+1})^{-1},
f_n ^{-1} \big|\overline{Z^*_2(\tilde{s})}  \right) + \varrho_2
\left(f^s_{n+1}, f_n \big|\overline{Z^*_1(\tilde{s} )} \right) <
(n+1)^{-1}.$$ The step with $ q = n+1$ is finished.

Suppose that functions $\varphi^p_{n+1, j}$ and $\psi^p_{n+1, j}$,
a \ndc subset $Z^*_j(s)$ of $Z^*_j(\tilde{s} )$, and a
homeomorphism $f^s_{n+1}: \overline{Z^*_1(\tilde{s} )} \rightarrow
\overline{Z^*_2(\tilde{s} )}$ have been obtained for each $s \in
\omega^p_{n+1}$ with $p > q$ and $j \in \{1,2 \}$, where
$\tilde{s} = s{\upharpoonright}( \textrm{lh}(s)-1)$.

Fix $s \in \omega^q_{n+1}$. We distinguish two cases.

\textit{Case} 1. Let $s = \tilde{s} \, \hat{\,} 0$. We can
construct the required sets and functions as above for the step $q
=n+1$.

\textit{Case} 2. Let $s = \tilde{s} \, \hat{\,} t$ and $t \geq 1$.
For each $i < t$ we have $\nu(\tilde{s} \, \hat{\,} i) \leq n$;
hence, there exists a unique number $\iota\in \omega$ such that
the tuple $\tilde{s} \, \hat{\,} i  \, \hat{\,} \iota \in
\omega^{q+1}_{n+1}$. By the inductive assumption, the sets
$Z^*_j(\tilde{s}\, \hat{\,} i)$  and the homeomorphism
$f^{\tilde{s}\, \hat{\,} i \, \hat{\,} \iota }_{n+1}:
\overline{Z^*_1(\tilde{s}\, \hat{\,} i)} \rightarrow
\overline{Z^*_2(\tilde{s}\, \hat{\,} i)}$ are already defined for
any $j \in \{1,2 \}$ and $i < t$ such that
$$\varrho_1 \left((f^{\tilde{s}\, \hat{\,} i \, \hat{\,}
\iota }_{n+1})^{-1}, f_n ^{-1} \big|\overline{Z^*_2(\tilde{s}\,
\hat{\,} i)} \right) + \varrho_2 \left(f^{\tilde{s}\, \hat{\,} i
\, \hat{\,} \iota }_{n+1}, f_n \big|\overline{Z^*_1(\tilde{s}\,
\hat{\,} i)} \right) < (n+1)^{-1}.$$

Using (d3), we observe that the sets $\overline{Z^*_j(\tilde{s}\,
\hat{\,} i)}$, where $0 \leq i < t$, are pairwise disjoint. From
(d4) it follows that $Y_j(t) = \bigcup \{
\overline{Z^*_j(\tilde{s}\, \hat{\,} i)}: i < t \}$ is a \ndc
subset of $\overline{Z^*_j(\tilde{s})}$. For a point $x \in
\overline{Z^*_1(\tilde{s}\, \hat{\,} i)}$ put $g^t_{n+1}(x)=
f^{\tilde{s}\, \hat{\,} i \, \hat{\,} \iota }_{n+1}(x)$. Then the
mapping $g^t_{n+1}: Y_1(t) \rightarrow Y_2(t)$ is a homeomorphism.
From (d8) and (d9) it follows that $f_n(Y_1(t)) = Y_2(t)$. By
construction,
$$\varrho_1 \left((g^t_{n+1})^{-1}, f_n ^{-1} \big| Y_2(t)
\right) + \varrho_2 \left(g^t_{n+1}, f_n \big|Y_1(t) \right) <
(n+1)^{-1}.$$

Fix $\alpha \in k^{q-1}$. By (d1), $\overline{Z_j(\tilde{s},
\alpha)}$ is a clopen subset of
$$\overline{X_j \left(\varphi^{q-1}_{n-t-1, j}(\tilde{s}, \alpha),
\psi^{q-1}_{n-t-1, j}(\tilde{s}, \alpha) \right)}$$ for $j \in
\{1, 2\}$. Choose the smallest number $m_{\tilde{s}, \alpha ,j}$
such that the set
$$\overline{Z_j(\tilde{s}, \alpha)} \bigcap \left(
\overline{X^*_j(\varphi^q_{n-t, j}(\tilde{s},
\alpha)\hat{\,}m_{\tilde{s}, \alpha, j})} \setminus Y_j(t)
\right)$$ is nonempty. Using (d7) and (d3), we see that the last
set is clopen in $\overline{X^*_j(\varphi^q_{n-t, j}(\tilde{s},
\alpha)\hat{\,}m_{\tilde{s}, \alpha, j})}$. Take a discrete clopen
(in $\overline{Z_j(\tilde{s}, \alpha)}$) cover $\tau_{\alpha, j}$
of $\overline{Z_j(\tilde{s}, \alpha)}$ such that

\begin{enumerate}[\upshape (e1)]
    \item $\mathrm{mesh}( \tau_{\alpha, j}) < (2n+2)^{-1}$;
    \item every $U \in \tau_{\alpha, j}$ intersects at most one set
    of the family $$\left\{X_j \left(\varphi^{q-1}_{n-t-1, j}(\tilde{s},
        \alpha)\hat{\,}m_{\tilde{s}, \alpha ,j}, \psi^{q-1}_{n-t-1,
        j} (\tilde{s}, \alpha)\hat{\,}\alpha_{q} \right):
        \alpha_{q} \in k  \right\};$$
    \item for every $U \in \tau_{\alpha, j}$ if $U \cap Y_j(t) \neq
    \emptyset$, then $$U \cap
        \left( \overline{X^*_j \left(\varphi^{q-1}_{n-t-1, j}
        (\tilde{s},\alpha)\hat{\,}m_{\tilde{s}, \alpha, j} \right)}
         \setminus Y_j(t) \right)  = \emptyset.$$
\end{enumerate}

According to (d8), we may assume that $ \tau_{\alpha, 2} = f_n(
\tau_{\alpha, 1})$. Let
\begin{multline*} \tau^*_{\alpha, 1} = \{U \in
\tau_{\alpha, 1}: U \cap \left( X^*_1(\varphi^{q-1}_{n-t-1,
1}(\tilde{s}, \alpha)\hat{\,}m_{\tilde{s}, \alpha ,1}) \setminus
Y_1(t) \right)
\neq \emptyset   \\
\quad \text{or }   \quad\  f_n(U) \bigcap \left( X^*_2
\left(\varphi^{q-1}_{n-t-1, 2}(\tilde{s},
\alpha)\hat{\,}m_{\tilde{s}, \alpha ,2} \right)  \setminus Y_2(t)
\right) \neq \emptyset \}.
\end{multline*}

Using weight-homogeneity of $X^*_1 \left(\varphi^{q-1}_{n-t-1,
1}(\tilde{s}, \alpha)\hat{\,}m_{\tilde{s}, \alpha ,1} \right)$, we
can index the family $\tau^*_{\alpha, 1}$ as $\{U_1(\gamma):
\gamma \in k \}$. Put $$\tau^*_{\alpha, 2} = f_n(\tau^*_{\alpha,
1}) = \{U_2(\gamma) = f_n(U_1(\gamma)): \gamma \in k \}.$$ Since
$f_n(Y_1(t)) = Y_2(t)$, for $U \in \tau_{\alpha, 1}$ we have $U
\cap Y_1(t) \neq \emptyset \iff f_n(U)\cap Y_2(t) \neq \emptyset$.
Together with (e3), this implies that $f_n(U) \cap Y_2(t) =
\emptyset$ if $U \in \tau^*_{\alpha, 1}$. Similarly, $ f^{-1}_n(U)
\cap Y_1(t) = \emptyset$ if $U \in \tau^*_{\alpha, 2}$.

Fix a $\gamma \in k$. For $j \in \{1, 2 \}$ choose the smallest
number $l_{\gamma, j} \in \omega$ such that the set
$$Z_j(s, \alpha \hat{\,} \gamma) = U_j(\gamma) \bigcap
\left( X_j(\varphi^{q-1}_{n-t-1, j}(\tilde{s}, \alpha
)\hat{\,}l_{\gamma, j}, \psi^{q-1}_{n-t-1, j}(\tilde{s}, \alpha
)\hat{\,} \alpha_{\gamma, j})  \setminus Y_j(t) \right)$$ is
nonempty for some $\alpha_{\gamma, j} \in k$. Then the tuples
$$\varphi^{q-1}_{n-t-1, j}(\tilde{s}, \alpha )\hat{\,}l_{\gamma, j}
= \varphi^{q}_{n+1, j}(s, \alpha \hat{\,} \gamma) \in \omega^{q}$$
and $$\psi^{q-1}_{n-t-1, j}(\tilde{s}, \alpha )\hat{\,}
\alpha_{\gamma, j} = \psi^{q}_{n+1, j}(s, \alpha \hat{\,} \gamma)
\in k^{q}$$ are assigned to $(s, \alpha \hat{\,} \gamma)$. As
above, we have $\nu(\varphi^{q}_{n+1, j}(s, \alpha \hat{\,}
\gamma)) \geq n+1.$

By construction, $Z_j(s, \alpha \hat{\,} \gamma)$ is a \ndc subset
of $Z_j(\tilde{s}, \alpha)$ and $\overline{Z_j(s, \alpha \hat{\,}
\gamma)}$ is a \ndc subset of $\overline{Z_j(\tilde{s}, \alpha)}$.
Since $\overline{Z_1(s, \alpha \hat{\,} \gamma)} \approx B(k)
\approx \overline{Z_2(s, \alpha \hat{\,} \gamma)}$, by
Lemma~\ref{eh}, there exists a homeomorphism $h_{n+1, \alpha
\hat{\,} \gamma}^s: U_1(\gamma) \rightarrow U_2(\gamma)$ such that
$$h_{n+1, \alpha \hat{\,} \gamma}^s \left(\overline{Z_1(s, \alpha
\hat{\,} \gamma)} \right) = \overline{Z_2(s, \alpha \hat{\,}
\gamma)}.$$

Clearly, the combination mapping
$$h^s_{n+1,\alpha} = \nabla \{h_{n+1, \alpha \hat{\,}
\gamma}^s: \gamma \in k \} : \cup \tau^*_{\alpha, 1} \rightarrow
\cup \tau^*_{\alpha, 2}$$ is a homeomorphism such that
$$h^s_{n+1,\alpha} \left(\cup \{\overline{Z_1(s, \alpha \hat{\,}
\gamma)}:  \gamma \in k \} \right) = \cup \{\overline{Z_2(s,
\alpha \hat{\,} \gamma)}: \gamma \in k \}.$$

By construction, the set $\cup \tau^*_{\alpha, j}$ is clopen in
$\overline{Z_j(\tilde{s}, \alpha)}$, $Y_j(t) \cap \bigcup
\tau^*_{\alpha, j} = \emptyset$, and $Y_j(t) \cap
\overline{Z_1(\tilde{s}, \alpha)}$ is a \ndc subset of
$\overline{Z_j(\tilde{s}, \alpha)}$ for each $j \in \{1, 2 \}$ and
$\alpha \in k^{q-1}$.

For $j \in \{1, 2 \}$ put $Z^*_j(s) = \cup \{Z_j(s, \alpha
\hat{\,} \gamma): \alpha \in k^{q-1}, \gamma \in k \}$. Then
$\overline{Z^*_j(s)} = \cup \{\overline{Z_j(s, \alpha \hat{\,}
\gamma)}: \alpha \in k^{q-1}, \gamma \in k \}$. Let $\tau^*_j =
\cup \{\bigcup \tau^*_{\alpha, j}: \alpha \in k^{q-1} \}$. Then
the set $\cup \tau^*_j$ is clopen in
$\overline{Z^*_j(\tilde{s})}$, $Y_j(t) \cap \bigcup \tau^*_j =
\emptyset$, and $Y_j(t) \cap \overline{Z^*_j(\tilde{s})}$ is a
\ndc subset of $\overline{Z^*_j(\tilde{s})}$.

We define the homeomorphism $h^s_{n+1}: \cup \tau^*_1 \rightarrow
\cup \tau^*_2$ by the rule $h^s_{n+1}(x) = h^s_{n+1,\alpha}(x)$ if
$x \in \cup \tau^*_{\alpha, 1}$ for some $\alpha \in k^{q-1}$. By
virtue of Lemma~\ref{q3} the homeomorphism $g^t_{n+1}: Y_1(t)
\rightarrow Y_2(t)$ can be extended to a homeomorphism
$$\tilde{g}\, ^t_{n+1}: \overline{Z^*_1(\tilde{s})} \setminus
\bigcup \tau^*_1 \rightarrow \overline{Z^*_2(\tilde{s})} \setminus
\bigcup \tau^*_2.$$ Then the combination mapping
$$f^s_{n+1} = h_{n+1}^s \nabla \tilde{g}\, ^t_{n+1}:
\overline{Z^*_1(\tilde{s})} \rightarrow \overline{Z^*_2(\tilde{s}
 )}$$ is a homeomorphism.
One can check that the following inequality
$$\varrho_1 \left((f^s_{n+1})^{-1},
f_n ^{-1} \big|\overline{Z^*_2(\tilde{s})}  \right) + \varrho_2
\left(f^s_{n+1}, f_n \big|\overline{Z^*_1(\tilde{s} )} \right) <
(n+1)^{-1}$$ holds.

Note that the sets $\overline{Z^*_j(\tilde{s})}$ and
$\overline{Z^*_j(\tilde{r})}$ are disjoint for different tuples
$s$ and $r$ from $\omega^{q}_{n+1}$.

This completes the induction on $q$ for $s \in \omega^{q}_{n+1}$.

The set $\omega^1_{n+1}$ contains a unique tuple $s=(n)$.  Since
$\tilde{s} = \Lambda$, we have $Z^*_j(\Lambda) = X_j$ and
$\overline{Z^*_j(\Lambda)} = B_j$ for $j \in \{1, 2 \}$. Then the
mapping $f^{(n)}_{n+1}: \overline{Z^*_1(\Lambda)} \rightarrow
\overline{Z^*_2(\Lambda)}$ is the required homeomorphism $f_{n+1}:
B_1 \rightarrow B_2$.

The induction on $n$ is finished.

Let us construct a homeomorphism $f: B_1 \rightarrow B_2$ with $f(
X_1) = X_2$.

Consider a point $x \in B_1$. The sequence $\{f_n(x): n \in \omega
\}$ is a Cauchy sequence in a complete metric space $B_2$. Hence,
it converges to a point $y \in B_2$; put $f(x) = y$. From (d10) it
follows that the sequence $\{f_n\}$ is uniformly convergent to the
mapping $f$. According to \cite[Theorem 4.2.19]{Eng}, $f$ is
continuous. Similarly, the mapping $g = \lim_{n \to \infty}
{f^{-1}_n}: B_2 \rightarrow B_1$ is continuous.

For a point $y \in B_2$ and $n \in \omega$ let $z = f^{-1}_n(y)$;
then
$$\begin{array}{l} \varrho_2(f \circ f^{-1}_n(y), y) =
  \lim_{i \to \infty} {\varrho_2(f_i \circ f^{-1}_n(y),
  f_n \circ f^{-1}_n(y)) } \\   \qquad{} = \lim_{i \to \infty}
   {\varrho_2(f_i(z), f_n(z)) }  \leq  (n+1)^{-1}.
\end{array}$$

Therefore, $\lim_{n \to \infty} {\varrho_2(f \circ f^{-1}_n(y), y)
= 0}$. But $\lim_{n \to \infty} {(f \circ f^{-1}_n)} = f \circ g$,
so, $f \circ g : B_2 \rightarrow B_2$ is the identity mapping. In
the same way, $g \circ f : B_1 \rightarrow B_1$ is the identity
mapping. Hence, $f^{-1} = g$ and $f: B_1 \rightarrow B_2$ is a
homeomorphism.

Let us show that $f(X_1) \subseteq X_2$. Take a point $x \in X_1$,
then $\{x\} = \cap \{X_1(\chi{ \upharpoonright}n,
\xi{\upharpoonright}n): n \in \omega \}$ for some $\chi \in
\omega^\omega$ and $\xi \in k^\omega $. From (d2), (c1), and (c4)
it follows that there exist $\tau \in \omega^\omega$ and
$\vartheta \in k^\omega $ such that $f_n(x) \in
\overline{X_2(\tau{ \upharpoonright}n,
\vartheta{\upharpoonright}n)}$ for each $n \in \omega$. Using
(c7), we get $\cap \{X_2(\tau{ \upharpoonright}n,
\vartheta{\upharpoonright}n) : n \in \omega \} = \{y \}$ for some
$y \in X_2$. The condition (c6) of Lemma~\ref{ad} implies that
$$f(x) = \lim_{n \to \infty} {f_n(x)} = y \in X_2.$$

Similarly, $f^{-1}(X_2) \subseteq X_1$. Thus, $f(X_1) = X_2$.
\end{proof}

\begin{thm}\label{t-qR}
Let $X$ be an absolute $F_{\sigma \delta}$-set of first category
such that $\mathrm{Ind}X= 0$, $X$ is nowhere $\sigma LW({<}k) +
\mathscr{G}_{\delta \sigma}$,  and $w(X)=k$. Then $X \approx
Q(k)^\omega$.
\end{thm}

\begin{proof}
In the case $k = \omega$ the theorem was proved by van
Engelen~\cite{EngQ}. Below we shall assume that $k > \omega$.

Clearly, the space $Q(k)^\omega$ is of first category,
$Q(k)^\omega \in \mathscr{F}_{\sigma \delta}$, and $Q(k)^\omega
\in \mathscr{E}_k$.

Let us verify that $Q(k)^\omega$ is nowhere $\sigma LW({<}k) +
\mathscr{G}_{\delta \sigma}$. Assume the converse. Then there
exists a nonempty clopen subset $V$ of $Q(k)^\omega$ such that $V
= \cup \{F_i \bigcup G_i : i \in \omega \}$, where all $F_i$ is
locally of weight ${<}k$ and $G_i \in \mathscr{G}_{\delta}$. Since
$V \in \mathscr{E}_k$, every $F_i$ is nowhere dense in $V$. Every
$G_i$ is nowhere dense in $V$, because $V$ is of first category.

By Corollary~\ref{c-hQk}, the space $Q(k)$ is $h$-homogeneous. By
the definition of the Tychonoff topology, $V$ contains an element
of the canonical base for the Cartesian product $U = \prod \{Q_i:
i \in \omega \}$, where each $Q_i \approx Q(k)$. Since the set
$F_0 \cup G_0$ is nowhere dense in $U$, there exists a nonempty
clopen subset $U_0 \subseteq U$ with $U_0 \cap (F_0 \cup G_0) =
\emptyset$. Choose $n_0 \in \omega$ and points $q_i \in Q_i$,
where $0 \leq i \leq n_0$, such that the closed set $$Z_0 = \{q_0
\}\times \ldots \times \{ q_{n_0} \} \times \prod \{Q_i: i
> n_0 \} \subset U_0.$$ Evidently, $Z_0 \approx Q(k)^\omega$.
Since $Z_0 \cap F_1$ is locally of weight ${<}k$, it is a nowhere
dense subset of $Z_0$. If the intersection $Z_0 \cap G_1$ is
nonempty, it is topologically complete as closed subset of $G_1$.
Then $Z_0 \cap G_1$ is nowhere dense in $Z_0$, because $Z_0$ is of
first category. So, $Z_0 \cap (F_1 \cup G_1)$ is a nowhere dense
subset of $Z_0$. Hence, there exist a number $n_1 \in \omega$,
points $q_i \in Q_i$, where $n_0 < i \leq n_1$, and a clopen set
$U_1 \subset U_0$ such that $U_1 \cap (F_1 \cup G_1) = \emptyset$
and the closed set $$Z_1 = \{q_0 \} \times \ldots \{q_{n_1}\}
\times \prod \{Q_i: i > n_1 \} \subseteq U_1 .$$ Continuing in
this way, we obtain the point $q = \{q_0, q_1, \ldots \} \in U$
which belongs to $$U \setminus \bigcup \{F_i \cup G_i : i \in
\omega \}.$$ Hence, $V \neq \cup \{F_i \bigcup G_i : i \in \omega
\}$, a contradiction. Thus, $Q(k)^\omega$ is nowhere $\sigma
LW({<}k) + \mathscr{G}_{\delta \sigma}$.

Consider the spaces $X$ and $Q(k)^\omega$ as dense subsets of
$B(k)$. By virtue of Theorem~\ref{tqc}, there exists a
homeomorphism $f: B(k) \rightarrow B(k)$ such that $f(X) =
Q(k)^\omega$.
\end{proof}

\begin{cor}\label{c-dhm}
The Baire space $B^*(k)$ is densely homogeneous with respect to
the space $Q(k)^\omega$.
\end{cor}

\begin{proof}
Van Engelen~\cite{En-tQ} proved that the Cantor set $\mathscr{C}$
is densely homogeneous with respect to the space $Q^\omega$. In
the case $k > \omega$ the corollary follows from
Theorems~\ref{tqc} and~\ref{t-qR}.
\end{proof}

\begin{cor}\label{c-a2}
Let $X_1$ and $X_2$ be dense subsets of the Baire space $B^*(k)$
such that $B^*(k) \setminus X_1 \approx B^*(k) \setminus X_2
\approx Q(k)^\omega$. Then $X_1 \approx X_2$.
\end{cor}

\begin{proof}
By Corollary~\ref{c-dhm} there exists a homeomorphism $f: B^*(k)
\rightarrow B^*(k)$ such that $f(B^*(k) \setminus X_1) = B^*(k)
\setminus X_2$. Then $f(X_1) = X_2$.
\end{proof}

\begin{cor}\label{c-hko}
Let $F$ be an $F_{\sigma \delta}$-subset of the Baire space
$B^*(k)$. Then $F \times Q(k)^\omega \approx Q(k)^\omega$ and $F$
is homeomorphic to a closed subset of $Q(k)^\omega$.
\end{cor}

\begin{proof}
One can verifies that the product $X = F \times Q(k)^\omega$
satisfies the conditions of Theorem~\ref{t-qR}. Then $X  \approx
Q(k)^\omega$.  Clearly, $F$ is homeomorphic to a closed subset of
$F \times Q(k)^\omega \approx Q(k)^\omega$.
\end{proof}

\section{Some homogeneous $F_{\sigma \delta}$- and $G_{\delta
\sigma}$-subsets of $B(k)$}

In this Section we give  topological characterizations of the
products $Q(\tau) \times Q(k)^\omega$ and their dense complements
$A^{II}_{\tau, k}$ in the Baire space $B(\tau)$, where $ \omega
\leq k \leq \tau$. We also investigate the products $Q(\tau)
\times A^{II}_{k, k}$. All these spaces are $h$-homogeneous.
Notice that this list of $h$-homogeneous $F_{\sigma \delta}$- and
$G_{\delta \sigma}$-subsets of $B(\tau)$ is not complete.

Theorems~\ref{t-vqx} and~\ref{t-vqx2} are Hurewicz-type theorems
for $F_{\sigma \delta}$-subsets of the Baire space $B(k)$.

To establish these results we shall use the following statements.

\begin{thm}~\cite[Theorem 3.5]{msu}\label{t-msu}
Let $Y$ be an $h$-homogeneous space and $X$ be a
weight-homogeneous space of weight $\tau$. Suppose $X$ is of first
category, $X \in \sigma LF(Y)$, and every nonempty clopen subset
of $X$ contains a \ndc copy of $Y$. Then $X$ is homeomorphic to
$Q(\tau) \times Y$.
\end{thm}

\begin{thm}\label{tm}~\cite[Theorem 3.13]{msu} or~\cite[Theorem
4]{Med10} Suppose $X$ and $Y$ are $h$-homogeneous spaces of first
category such that $w(X) = w(Y)$, $Y \in \sigma LF(X)$, and $X \in
\sigma LF(Y)$. Then $X$ is homeomorphic to $Y$.
\end{thm}

The last theorem implies the following corollary.

\begin{cor}\label{c-xhq}
Let $X$ be an $h$-homogeneous spaces of first category and $w(X)
\geq k$. Then $X \approx Q  \times X \approx Q(k)  \times X$.
\end{cor}

\begin{thm}\label{t-emb}
Let $\omega < k \leq \tau$. Let the Baire space $B(k)$ be densely
homogeneous with respect to an $h$-homogeneous space $Y$. Then the
Baire space $B(\tau)$ is densely homogeneous with respect to the
product $Q(\tau) \times Y$.
\end{thm}

Van Engelen (see~\cite[ Theorem 4.7]{En-B2}) proved that if the
Cantor set $\mathscr{C}$ is densely homogeneous with respect to an
$h$-homogeneous space $A$, then $\mathscr{C}$ is densely
homogeneous with respect to the product $Q \times A$.
Theorem~\ref{t-emb} generalize this result for the non-separable
case. It was announced in~\cite{M-f86} and proved in~\cite{msu}
(see \cite[Lemma 5.3]{msu}).

\begin{defin}\label{d-1}
In this Section we shall often denote the space $Q(k)^\omega$ by
$M^I_k$. Let us introduce the space $A^{II}_k$. Fix a dense copy
$A$ of $Q(k)^\omega$ in $B^*(k)$ and put $A^{II}_k = B^*(k)
\setminus A$. Let $A^I_k = Q \times A^{II}_k$ and $M^{II}_k =
B^*(k) \setminus Y$, where $Y$ is a dense copy of $A^I_k$ in
$B^*(k)$.

Notice that the spaces $M^I_k$ and $A^I_k$  are of first category.
The spaces $M^{II}_k$ and $A^{II}_k$ contain  dense topologically
complete subspaces. Clearly, $M^I_k$ and $M^{II}_k$ are absolute
$\mathscr{F}_{\sigma \delta }$-sets. The spaces $A^I_k$ and
$A^{II}_k$ are absolute $G_{\delta \sigma}$-sets. Furthermore,
each of the spaces $M^I_k$, $M^{II}_k$, $A^I_k$, and $A^{II}_k$ is
a weight-homogeneous space of weight $k$.
\end{defin}

\begin{thm}\label{t-aR}
Let $X$ be an absolute $G_{\delta \sigma}$-set such that
$\mathrm{Ind}X= 0$, $X$ is nowhere $\mathscr{F}_{\sigma \delta }
\setminus \sigma LW({<}k)$, $X$ contains a dense topologically
complete subspace, and $w(X)=k$. Then $X \approx A^{II}_k$.
\end{thm}

\begin{proof}
Embed $X$ densely in $B^*(k)$. We claim that $B^*(k) \setminus X$
is nowhere $\sigma LW({<}k) + \mathscr{G}_{\delta \sigma}$. Assume
the converse. Then there exists a nonempty clopen subset $U
\subset B^*(k)$ such that $U \setminus X = G \cup L$, where $G$ is
a $G_{\delta \sigma}$-subset of $B^*(k)$ and $L \in  \sigma
LW({<}k)$. Without loss of generality, $G \cap L = \emptyset$.
Then $$U \cap X = (U \setminus G) \setminus L \in
\mathscr{F}_{\sigma \delta } \setminus \sigma LW({<}k),$$ a
contradiction. Next, $B^*(k) \setminus X$ is of first category.
From Theorem~\ref{t-qR} it follows that $B^*(k) \setminus X
\approx Q(k)^\omega $. According to Definition~\ref{d-1},
$A^{II}_k = B^*(k) \setminus A$ with $A \approx Q(k)^\omega$. By
Corollary~\ref{c-a2}, $X \approx A^{II}_k$.
\end{proof}

\begin{cor}\label{c-ha12t}
$A^I_k$, $A^{II}_k$, and $Q(\tau) \times A^I_k \approx Q(\tau)
\times A^{II}_k$  are $h$-homogeneous spaces if $\omega \leq k
\leq \tau$.
\end{cor}

\begin{proof}
$h$-Homogeneity of $A^{II}_k$ follows from Theorem~\ref{t-aR}. The
spaces $A^I_k = Q \times A^{II}_k$  is $h$-homogeneous as the
product of two $h$-homogeneous spaces (see \cite[Lemma 2.3]{msu}).
Similarly, $Q(\tau) \times A^{II}_k$ is $h$-homogeneous. Since
$Q(\tau) \approx Q(\tau)  \times Q$, we have $Q(\tau) \times A^I_k
\approx Q(\tau) \times A^{II}_k$.
\end{proof}

\begin{thm}\label{t-dha2}
The Baire space $B^*(k)$ is densely homogeneous with respect to
the space $A^{II}_k$.
\end{thm}

\begin{proof}
Let $X_1$ and $X_2$ be two dense copies of $A^{II}_k$ in $B^*(k)$.
As in the proof of Theorem~\ref{t-aR} we obtain that $B^*(k)
\setminus X_1 \approx B^*(k) \setminus X_2 \approx Q(k)^\omega$.
By Corollary~\ref{c-dhm} there exists a homeomorphism $f: B^*(k)
\rightarrow B^*(k)$ such that $f(B^*(k) \setminus X_1) = B^*(k)
\setminus X_2$. Then $f(X_1) = X_2$.
\end{proof}

\begin{thm}\label{t-ema}
Let $\omega < k \leq \tau$. The Baire space $B(\tau)$ is densely
homogeneous with respect to the space $Q(\tau) \times A^I_k$ and
with respect to the space $Q(\tau) \times M^I_k$.
\end{thm}

\begin{proof}
Using Theorem~\ref{t-emb} and Corollary~\ref{c-dhm}, we obtain
that $B(\tau)$ is densely homogeneous with respect to the space
$Q(\tau) \times M^I_k$. By Corollary~\ref{c-ha12t} we have
$Q(\tau) \times A^I_k \approx Q(\tau) \times A^{II}_k$. From
Theorems~\ref{t-emb} and~\ref{t-dha2} it follows that $B(\tau)$ is
densely homogeneous with respect to $Q(\tau) \times A^I_k$.
\end{proof}

To give an internal description of the spaces $A^I_k$, $Q(\tau)
\times A^I_k$, $M^I_k$, and $Q(\tau) \times M^I_k$ we improve
Theorem~\ref{hf}.

\begin{thm}\label{t-vqx}
Let $X$ be an $F_{\sigma \delta}$-subset of the Baire space
$B^*(k)$ and $X \notin \sigma LW({<}k) + \mathscr{G}_{\delta
\sigma}$. Then there exists a closed subset $M \subset B^*(k)$
such that $M \cap X$ is a \ndc subset of $X$, $M \cap X \approx
M^I_k$, $M \setminus X$ is a \ndc subset of $B^*(k) \setminus X$,
$M \setminus X \approx A_k^{II}$,  and $M \approx B^*(k)$.
\end{thm}

\begin{proof}
By Theorem~\ref{hf} with $F = B^*(k)$, there exists a closed
subset $M_1$ of $B^*(k)$ such that $M_1 \approx B^*(k)$ and the
intersection $Y_1 = M_1 \cap X$ is nowhere $\sigma LW({<}k) +
\mathscr{G}_{\delta \sigma}$ and of first category. Clearly, $Y_1$
is an $F_{\sigma \delta}$-subset of $B^*(k)$. From
Theorem~\ref{t-qR} it follows that $Y_1 \approx Q(k)^\omega$.
Since $Y_1 \approx Q(k) \times Q(k)^\omega$, we can find a \ndc
subset $Y \subset Y_1$ with $Y \approx M^I_k$. Put $M =
\overline{Y}$. Clearly, $M \cap X = Y \approx M^I_k$.  Since $M
\approx B^*(k)$, $M \setminus Y$ is a dense $G_{\delta
\sigma}$-subset of $M$. From Theorem~\ref{t-qR} it follows that $M
\setminus Y$ is nowhere $\mathscr{F}_{\sigma \delta } \setminus
\sigma LW({<}k)$. Since $Y$ is of first category, $M \setminus Y$
contains a dense topologically complete subspace. By virtue of
Theorem~\ref{t-aR}, $M \setminus Y \approx A_k^{II}$. By
construction, $M \setminus Y$ is a closed subset of $B^*(k)
\setminus X$. Since $M$ is a \ndc subset of $M_1$ and $M_1
\setminus X$ is dense in $M_1$, we conclude that $M \setminus X =
M \setminus Y$ is a nowhere dense subset of $B^*(k) \setminus X$.
\end{proof}

From the last theorem we immediately obtain

\begin{cor}\label{c-cm2}
Let $X$ be an $F_{\sigma \delta}$-subset of the Baire space
$B^*(k)$ and $X \notin \sigma LW({<}k) + \mathscr{G}_{\delta
\sigma}$. Then $X$ contains a \ndc copy of $Q(k)^\omega$.
\end{cor}

\begin{cor}\label{c-f2}
Let $X$ and $Y$ be $F_{\sigma \delta}$-subsets of the Baire space
$B^*(k)$ and $X \notin \sigma LW({<}k) + \mathscr{G}_{\delta
\sigma}$. Then there exists a \ndc subset $F$ of $X$ such that $F
\approx Y$.
\end{cor}

\begin{proof}
By Corollary~\ref{c-cm2}, $X$ contains a \ndc subset $M$ with $M
\approx Q(k)^\omega$. From Corollary~\ref{c-hko} it follows that
$M$ contains a closed copy $F$ of $Y$.
\end{proof}

\begin{cor}\label{c-ca2}
Let $Y$ be a $G_{\delta \sigma}$-subset of the Baire space
$B^*(k)$ and $Y \notin \mathscr{F}_{\sigma \delta } \setminus
\sigma LW({<}k)$. Then $Y$ contains a \ndc copy of $A_k^{II}$.
\end{cor}

\begin{proof}
The corollary follows from Theorem~\ref{t-vqx} with $X = B^*(k)
\setminus Y$.
\end{proof}

\begin{rem}\label{r-1}
In~\cite[Section~5]{Med10} it was constructed \textit{a canonical
element} $M_3(k)$ for $F_{\sigma \delta}$-subsets of the Baire
space $B(k)$. This means (see~\cite{Med10}) that $M_3(k)$ is an
$h$-homogeneous absolute $F_{\sigma \delta}$-set of first
category, $M_3(k)$ contains a closed copy of any $F_{\sigma
\delta}$-subset of $B(k)$, and $w(M_3(k)) = k$. In particulary,
$M_3(k)$ contains a closed copy of $Q(k)^\omega$.
Corollary~\ref{c-hko} shows that the space $Q(k)^\omega$ contains
a closed copy of $M_3(k)$. Theorem~\ref{t-msu} implies that
$M_3(k) \approx Q(k) \times Q(k)^\omega \approx Q(k)^\omega$.

Let $A$ be a dense subset of $B^*(k)$ such that $A \approx
M_3(k)$. By Definition~\ref{d-1} we have $A_k^{II} = B^*(k)
\setminus A$. In~\cite{Med10} the set $A_k^{II}$ was denoted by
$A_3(k)$. According to \cite[Lemma 6]{Med10}, every $G_{\delta
\sigma}$-subset of $B(k)$ is homeomorphic to a closed subset of
$B(k) \times A_k^{II}$. Using Theorem~\ref{t-aR}, we obtain that
$B(k) \times A_k^{II} \approx A_k^{II}$. Hence, every $G_{\delta
\sigma}$-subset of $B(k)$ is homeomorphic to a closed subset of
$A_k^{II}$. In particulary, $A_k^{II}$ contains a closed copy of
$A_k^I$.
\end{rem}

\begin{cor}\label{c-ea2}
Let $X$ and $Y$ be $G_{\sigma \delta}$-subsets of the Baire space
$B^*(k)$ and $X \notin \mathscr{F}_{\sigma \delta } \setminus
\sigma LW({<}k)$. Then there exists a \ndc subset $F$ of $X$ such
that $F \approx Y$.
\end{cor}

\begin{proof}
By Corollary~\ref{c-ca2}, $X$ contains a \ndc subset $M$ with $M
\approx A^{II}_k$. According to Remark~\ref{r-1}, $M$ contains a
closed copy $F$ of $Y$.
\end{proof}

\begin{thm}\label{t-ak1}
Let $X$ be an absolute $G_{\delta \sigma}$-set of first category
such that $\mathrm{Ind}X= 0$, $X$ is nowhere $\mathscr{F}_{\sigma
\delta } \setminus \sigma LW({<}k)$, and $w(X)=k$. Then $X \approx
A^I_k$.
\end{thm}

\begin{proof}
Since $A^I_k = Q \times A^{II}_k \approx Q \times A^I_k$, the
space $A^I_k$ contains a \ndc copy of itself.
Corollary~\ref{c-ea2} implies that each of the spaces $X$ and
$A^I_k$ contains a \ndc copy of another. By Theorem~\ref{t-msu},
$X \approx Q(k) \times A^I_k \approx A^I_k$.
\end{proof}

\begin{thm}\label{t-atk}
Let $\omega \leq k \leq \tau$. Let $X$ be an absolute $G_{\delta
\sigma}$-set of first category such that $X$ is nowhere
$\mathscr{F}_{\sigma \delta } \setminus \sigma LW({<}k)$, $X \in
\sigma LW({<}k^+)$, and $X \in \mathscr{E}_\tau$. Then $X \approx
Q(\tau) \times A_k^I$.
\end{thm}

\begin{proof}
Since $X \in \sigma LW({<}k^+)$, we have $X = \cup \{X_i: i \in
\omega \}$, where each $X_i$ is a closed subset of $X$ and locally
of weight ${\leq}k$. Using $\textrm{Ind} X =0$, we can assume that
$X_i \cap X_j = \emptyset$ whenever $i \neq j$.

We claim that every nonempty clopen subset $U$ of $X$ contains a
\ndc copy of $A_k^I$. If we assume that each $U \cap X_i = F_i
\setminus L_i$, where $F_i$ is an absolute $ F_{\sigma
\delta}$-set and $L_i \in \sigma LW({<}k)$, then $U \in
\mathscr{F}_{\sigma \delta } \setminus \sigma LW({<}k)$
contradicting the conditions of the theorem. Hence, there exists a
$j \in \omega$ such that $U \cap X_j$ is not $\mathscr{F}_{\sigma
\delta } \setminus \sigma LW({<}k)$. Clearly, $U \cap X_j$ is an
absolute $G_{\delta \sigma }$-set. By Corollary~\ref{c-ea2}, $U
\cap X_j$ contains a \ndc copy of $A^I_k$. Then $U$ is the same.
The claim is proved. \hfill $\lozenge$

From Corollary~\ref{c-ea2} it follows that each $X_i \in
LF(A^I_k)$. Hence, $X \in \sigma LF(A^I_k)$. Theorem~\ref{t-msu}
implies that $X \approx Q(\tau) \times A_k^I$.
\end{proof}

\begin{thm}\label{t-mk2}
Let $X$ be an absolute $F_{\sigma \delta }$-set such that
$\mathrm{Ind}X= 0$, $X$ is nowhere $\sigma LW({<}k) +
\mathscr{G}_{\delta \sigma}$, $X$ contains a dense topologically
complete subspace, and $w(X)=k$. Then $X \approx M^{II}_k$.
\end{thm}

\begin{proof}
In the case $k=\omega$ the theorem was proved by van Engelen (see
\cite[Theorem 4.4]{EngQ}).

Let $k > \omega$. By Definition~\ref{d-1}, $M^{II}_k = B(k)
\setminus Y$, where $Y$ is a dense copy of $A^I_k$ in $B(k)$.
Embed $X$ densely in $B(k)$. As above, we obtain that $B(k)
\setminus X$ is nowhere $ \mathscr{F}_{\sigma \delta } \setminus
\sigma LW({<}k)$. Clearly, $B(k) \setminus X$ is of first category
and dense in $B(k)$. From Theorem~\ref{t-ak1} it follows that
$B(k) \setminus X \approx A_k^I$. According to Theorem~\ref{t-ema}
with $k = \tau$, the Baire space $B(k)$ is densely homogeneous
with respect to the space $Q(k) \times A^I_k$. By
Corollary~\ref{c-xhq}, $Q(k) \times A^I_k \approx A^I_k$. Then
there exists a homeomorphism $f:B(k) \rightarrow B(k)$ such that
$f(B(k) \setminus X) = B(k) \setminus Y$. Hence, $f(X) =M^{II}_k$.
\end{proof}

\begin{cor}
$M^{II}_k$ and $Q(\tau) \times M^{II}_k$ are $h$-homogeneous
spaces if $\omega \leq k \leq \tau$.
\end{cor}

\begin{proof}
$h$-Homogeneity of $M^{II}_k$ follows from Theorem~\ref{t-mk2}.
The space $Q(\tau) \times M^{II}_k$ is $h$-homogeneous as the
product of two $h$-homogeneous spaces (see \cite[Lemma 2.3]{msu}).
\end{proof}

\begin{cor}
\rm{1)} $Q \times M^{II}_k$ is homeomorphic to $Q(k)^\omega$.

\rm{2)} $Q(\tau) \times M^{II}_k$ is homeomorphic to $Q(\tau)
\times Q(k)^\omega$ if $\omega \leq k \leq \tau$.
\end{cor}

\begin{proof}
The first part of the corollary follows from Theorems~\ref{t-mk2}
and~\ref{t-qR}. The second part follows from the first part.
\end{proof}

\begin{thm}\label{t-qtk}
Let $\omega \leq k \leq \tau$. Let $X$ be an absolute
$\mathscr{F}_{\sigma \delta }$-set of first category such that $X$
is nowhere $\sigma LW({<}k) +G_{\delta \sigma}$, $X \in \sigma
LW({<}k^+)$, and $X \in \mathscr{E}_\tau$. Then $X \approx Q(\tau)
\times M_k^I$.
\end{thm}

\begin{proof}
Since $X \in \sigma LW({<}k^+)$, we have $X = \cup \{X_i: i \in
\omega \}$, where each $X_i$ is a closed subset of $X$ and locally
of weight ${\leq}k$.

We claim that every nonempty clopen subset $U$ of $X$ contains a
\ndc copy of $M_k^I$. If we assume that each $U \cap X_i = G_i
\cup L_i$, where $G_i$ is an absolute $ G_{\delta \sigma }$-set
and $L_i \in \sigma LW({<}k)$, then $U \in \sigma LW({<}k) +
\mathscr{G}_{\delta \sigma}$ contradicting the conditions of the
theorem. Hence, there exists a $j \in \omega$ such that $U \cap
X_j$ is not $\sigma LW({<}k) +G_{\delta \sigma}$. Clearly, $U \cap
X_j$ is an absolute $\mathscr{F}_{\sigma \delta }$-set. By
Corollary~\ref{c-cm2}, $U \cap X_j$ contains a \ndc copy of
$M^I_k$. Then $U$ is the same. The claim is proved. \hfill
$\lozenge$

From Corollary~\ref{c-hko} it follows that each $X_i \in
LF(M^I_k)$. Hence, $X \in \sigma LF(M^I_k)$. Theorem~\ref{t-msu}
implies that $X \approx Q(\tau) \times M_k^I$.
\end{proof}

\begin{thm}\label{t-dhm2}
The Baire space $B^*(k)$ is densely homogeneous with respect to
the space $M^{II}_k$.
\end{thm}

\begin{proof}
In the case $k=\omega$ the theorem was proved by van Engelen (see
\cite[Theorem 4.4]{EngQ}).

Let $k > \omega$. Consider two dense copies $X_1$ and $X_2$  of
$M^{II}_k$ in $B(k)$. As in the proof of Theorem~\ref{t-mk2} we
obtain that $B(k) \setminus X_1 \approx B(k) \setminus X_2 \approx
A_k^I$. By Theorem~\ref{t-ema} with $k = \tau$, the Baire space
$B(k)$ is densely homogeneous with respect to the space $Q(k)
\times A^I_k \approx A^I_k$. Hence, there exists a homeomorphism
$f: B(k) \rightarrow B(k)$ such that $f(B(k) \setminus X_1) = B(k)
\setminus X_2$. Then $f(X_1) = X_2$.
\end{proof}

The following theorem is a modification of Theorem~\ref{t-vqx}.

\begin{thm}\label{t-vqx2}
Let $X$ be an $F_{\sigma \delta}$-subset of the Baire space
$B^*(k)$ and $X \notin \sigma LW({<}k) + \mathscr{G}_{\delta
\sigma}$. Then there exists a closed subset $M \subset B^*(k)$
such that $M \cap X$ is a \ndc subset of $X$, $M \cap X \approx
M^{II}_k$, $M \setminus X$ is a \ndc subset of $B^*(k) \setminus
X$, $M \setminus X \approx A_k^I$,  and $M \approx B^*(k)$.
\end{thm}

\begin{proof}
To proof the theorem we repeat the proof of Theorem~\ref{t-vqx}
with the subset $Y \approx M^{II}_k $ of $Y_1$ instead of $Y
\approx M^I_k$. We can do this, because the space $Q(k)^\omega$
contains a \ndc copy of $M^{II}_k$.
\end{proof}

\begin{defin}\label{d-2}
For given cardinals $k$ and $\tau$, where $\omega \leq k \leq
\tau$, fix a dense copy $A$ of $Q(\tau) \times M^I_k$ in the Baire
space $B^*(\tau)$. Denote $B^*(\tau) \setminus A$ by
$A^{II}_{\tau, k}$. Put $M^{II}_{\tau, k} = B^*(\tau) \setminus
Y$, where $Y$ is a dense copy of $Q(\tau) \times A^I_k$ in
$B^*(\tau)$.

Notice that $A^{II}_{k, k} = A^{II}_k$ and $M^{II}_{k, k} =
M^{II}_k$.
\end{defin}

\begin{thm}\label{t-1tk}
Let $\omega < k \leq \tau$. Let $X$ be an absolute $G_{\delta
\sigma}$-set such that $X$ contains a dense topologically complete
subspace, $X$ is nowhere $\mathscr{F}_{\sigma \delta } \setminus
\sigma LW({<}k) $, $X \in \mathscr{G}_{\delta} \setminus \sigma
LW({<}k^+)$, and $X \in \mathscr{E}_\tau$. Then $X \approx
A^{II}_{\tau, k}$.
\end{thm}

\begin{proof}
Since $X \in \mathscr{E}_\tau$, we can embed $X$ in $B(\tau)$ as a
dense $G_{\delta \sigma}$-subset. From $X \in \mathscr{G}_{\delta}
\setminus \sigma LW({<}k^+)$ it follows that $X = G \setminus L$,
where $G$ is a dense $G_\delta$-subset of $B(\tau)$ and $L \in
\sigma LW({<}k^+)$. Then $G \approx B(\tau)$. Since $X$ is nowhere
$\mathscr{F}_{\sigma \delta} \setminus \sigma LW({<}k)$, we see
that the set $Z = G \setminus X$  is nowhere $\sigma LW({<}k) +
\mathscr{G}_{\delta \sigma}$. Clearly, $Z$ is dense in $G$ and of
first category.  By virtue of Theorem~\ref{t-qtk}, $Z \approx
Q(\tau) \times M^I_k$. From Theorem~\ref{t-ema} it follows that
there exists a homeomorphism $f: G \rightarrow B(\tau)$ such that
$f(Z) = A$, where the set $A$ is given by Definition~\ref{d-2}.
Then $f(X) = A^{II}_{\tau, k}$.
\end{proof}

In such a way we can prove the following theorem.

\begin{thm}\label{t-2tk}
Let $\omega < k \leq \tau$. Let $X$ be an absolute $F_{\sigma
\delta}$-set such that $X$ contains a dense topologically complete
subspace, $X$ is nowhere $\mathscr{G}_{\delta \sigma } + \sigma
LW({<}k) $, $X \in \mathscr{G}_{\delta} \setminus \sigma
LW({<}k^+)$, and $X \in \mathscr{E}_\tau$. Then $X \approx
M^{II}_{\tau, k}$.
\end{thm}

Using Theorems~\ref{t-1tk} and~\ref{t-2tk}, we obtain

\begin{cor}
Let $\omega < k \leq \tau$. The spaces $A^{II}_{\tau, k}$ and
$M^{II}_{\tau, k}$ are $h$-homogeneous.
\end{cor}

\begin{rem}
If $k < \tau$, then $M^{II}_{\tau, k}$ is not homeomorphic to
$M^{II}_{\tau}$ and $A^{II}_{\tau, k}$ is not homeomorphic to
$A^{II}_{\tau}$.
\end{rem}

\begin{thm}
Let $X$ be a homogeneous absolute $G_{\delta \sigma}$-set of first
category such that $X \notin \mathscr{F}_{\sigma \delta }
\setminus \sigma LW({<}k)$, $X \in \sigma LW({<}k^+)$,  and $X \in
\mathscr{E}_\tau$, where $\omega \leq k \leq \tau$. Then $X
\approx Q(\tau) \times A^I_k \approx Q(\tau) \times A^{II}_k$.
\end{thm}

\begin{proof}
Assume that there exists a nonempty clopen subset $U \subset X$
with $U \in \mathscr{F}_{\sigma \delta } \setminus \sigma
LW({<}k_1)$ for some $k_1 \leq k$. Using homogeneity of $X$, we
can find a clopen cover $\mathscr{U}$ of $X$ consisting of sets
that are homeomorphic to $U$. Since $\mathrm{Ind}X=0$, there
exists a refinement $\mathscr{V}$ of $\mathscr{U}$ consisting of
pairwise disjoint nonempty clopen subsets. Then $$X = \cup
\mathscr{V} \in \mathscr{F}_{\sigma \delta } \setminus \sigma
LW({<}k_1) \subseteq \mathscr{F}_{\sigma \delta } \setminus \sigma
LW({<}k),$$ a contradiction with the choose of $k$. Hence, $X$ is
nowhere $\mathscr{F}_{\sigma \delta } \setminus \sigma LW({<}k)$.

From Theorem~\ref{t-atk} it follows that $X \approx Q(\tau) \times
A^I_k$. It remains to apply Corollary~\ref{c-ha12t}.
\end{proof}

In such a way we can prove the following theorem.

\begin{thm}
Let $X$ be a homogeneous absolute $\mathscr{F}_{\sigma \delta
}$-set of first category such that $X \notin \sigma LW({<}k)
+G_{\delta \sigma}$, $X \in \sigma LW({<}k^+)$, and $X \in
\mathscr{E}_\tau$, where $\omega \leq k \leq \tau$. Then $X
\approx Q(\tau) \times M_k^I \approx Q(\tau) \times M^{II}_k$.
\end{thm}

\end{document}